\documentclass[11pt,oneside,english]{amsart}
\usepackage[T1]{fontenc}
\usepackage[latin9]{inputenc}
\usepackage{geometry}
\usepackage{color}
\usepackage{babel}
\usepackage{amsthm}
\usepackage{amssymb}
\usepackage{stackrel}
\usepackage{setspace}
\usepackage[all]{xy}
\usepackage[unicode=true,pdfusetitle,
 bookmarks=true,bookmarksnumbered=false,bookmarksopen=false,
 breaklinks=false,pdfborder={0 0 0},backref=false,colorlinks=true]
 {hyperref}
\hypersetup{
 citecolor=blue}
\usepackage{breakurl}

\makeatletter

\newcommand*\LyXZeroWidthSpace{\hspace{0pt}}

\numberwithin{equation}{section}
\numberwithin{figure}{section}
\theoremstyle{plain}
\newtheorem{thm}{\protect\theoremname}[section]
  \theoremstyle{remark}
  \newtheorem{rem}[thm]{\protect\remarkname}
  \theoremstyle{plain}
  \newtheorem{lem}[thm]{\protect\lemmaname}
  \theoremstyle{definition}
  \newtheorem{defn}[thm]{\protect\definitionname}
  \theoremstyle{plain}
  \newtheorem{prop}[thm]{\protect\propositionname}
  \theoremstyle{plain}
  \newtheorem{cor}[thm]{\protect\corollaryname}
  \theoremstyle{definition}
  \newtheorem{example}[thm]{\protect\examplename}

\usepackage{ifpdf} 
\ifpdf 

 \IfFileExists{lmodern.sty}{\usepackage{lmodern}}{}

\fi 

\usepackage[figure]{hypcap}

\let\myTOC\tableofcontents
\renewcommand\tableofcontents{%
  \frontmatter
  \pdfbookmark[1]{\contentsname}{}
  \myTOC
  \mainmatter }

\usepackage{fancyhdr}

\makeatother

  \providecommand{\corollaryname}{Corollary}
  \providecommand{\definitionname}{Definition}
  \providecommand{\examplename}{Example}
  \providecommand{\lemmaname}{Lemma}
  \providecommand{\propositionname}{Proposition}
  \providecommand{\remarkname}{Remark}
\providecommand{\theoremname}{Theorem}

\begin{document}

\title{Framed symplectic sheaves on surfaces}

\author{Jacopo Scalise}

\dedicatory{\emph{SISSA, Trieste - jscalise@sissa.it}}
\begin{abstract}
A framed symplectic sheaf on a smooth projective surface $X$ is a
torsion-free sheaf $E$ together with a trivialization on a divisor
$D\subseteq X$ and a morphism $\Lambda^{2}E\rightarrow\mathcal{O}_{X}$
satisfying some additional conditions. We construct a moduli space
for framed symplectic sheaves on a surface, and present a detailed
study for $X=\mathbb{P}_{\mathbb{C}}^{2}$. In this case, the moduli
space is irreducible and admits an ADHM-type description and a birational proper map
onto the space of framed symplectic ideal instantons. 
\end{abstract}

\maketitle

\section{introduction}

Let $l$ be a fixed line in the complex projective plane $\mathbb{P}_{\mathbb{C}}^{2}$.
An \emph{$l-$framed sheaf of charge} $n$ is defined to be a pair
$(E,a)$ where $E$ is a torsion-free coherent sheaf on $\mathbb{P}_{\mathbb{C}}^{2}$
of generic rank $r$ and $c_{2}(E)=n$, and $a:E\mid_{l}\rightarrow\mathcal{O}_{l}^{\oplus r}$
is an isomorphism, which we call \emph{framing}. Framed sheaves are
parameterized by a fine moduli space, usually denoted $\mathcal{M}(r,n)$.
The interest in this space arises from Gauge theory. An $SU(r)-$instanton
on the four sphere with charge $n$ is a principal $SU(r)$-bundle
endowed with an anti-selfdual connection. As explained in \cite{ADHM},
instantons can be described by means of linear algebraic data. The
upshot here is an interpretation of the moduli space of instantons
as an algebro-geometric object, namely as a the moduli space $\mathcal{M}^{reg}(r,n)$
of holomorphic vector bundles on the projective plane with a framing
on a line, see \cite{Do}. There exists a partial compactification
$\mathcal{M}^{reg}(r,n)\subseteq\mathcal{M}_{0}(r,n)$; this is obtained
allowing instantons to degenerate to a so-called \emph{ideal instanton},
that is, a bundle with a connection whose square curvature density
has distributional degenerations at a finite number of points. This
space, which is also called \emph{Uhlenbeck space}, can be thought
of as an affine $\mathbb{C}-$scheme. It is very singular, and it
does not possess very well defined modular properties; however, there
exists a natural map 
\[
\pi:\mathcal{M}(r,n)\rightarrow\mathcal{M}_{0}(r,n)
\]
which is in fact a resolution of singularities. Both $\mathcal{M}_{0}(r,n)$
and $\mathcal{M}(r,n)$ contain an open subscheme isomorphic to $\mathcal{M}^{reg}(r,n),$
over which $\pi$ is an isomorphism. 

If we change $SU(r)$ for another real simple Lie group $G$, we have
$G-$analogues of $\mathcal{M}^{reg}(r,n)$ (we can take the moduli
space of framed principal $G_{\mathbb{C}}-$bundles) and also for
$\mathcal{M}_{0}(r,n)$ (see for example \cite{Bal,BFG,Ch,NS}), but
this time no modular desingularization of the Uhlenbeck space is known.
Indeed, it is not clear what should be the correct definition of ``weak
principal bundle'' to use in order to obtain a smooth partial compactification,
as one does with torsion-free sheaves in the classical case.

In this paper we focus on the case of the symplectic group $G=SP_{r}$.
We define a \emph{framed symplectic sheaf }to be a framed sheaf $(E,a)$
with a morphism $\varphi:\Lambda^{2}E\rightarrow\mathcal{O}_{\mathbb{P}^{2}}$
which behaves as a symplectic form on the locally free locus of $E$.
The main point of this paper is to prove that framed symplectic sheaves
admit a fine module space as well, denoted $\mathcal{M}_{\Omega}(r,n)$,
which comes with a proper birational morphism onto the symplectic
variant of the Uhlenbeck space.

The paper is organized as follows. Section 2 fixes some notation and
introduces the definition of the moduli functor for framed symplectic
sheaves. Moreover, some lemmas for later use are stated. Section 3
describes the construction of moduli space for a general smooth projective
surface $X$, where we fix a suitable framing divisor $D\subseteq X$,
and provides the computation of the tangent space at a point of the
moduli space. Section 4 specializes to the case $(X,D)=(\mathbb{P}^{2},l)$,
and presents an alternative definition of the moduli space by means
of linear data, in the spirit of the ADHM construction for the classical
case. In section 5 this description is applied in order to prove irreducibility
of the moduli space on $\mathbb{P}^{2}$. Section 6 deals with Uhlenbeck
spaces, and provides some results on the singularities of $\mathcal{M}_{\Omega}(r,n)$.

\subsection*{Acnkowledgments}

This paper covers a part of my PhD thesis. I am very indebdted with my supervisor Ugo Bruzzo for introducing
me to the subject and for his patient guidance through this work.
Also, I wish to thank Francesco Sala, Pietro Tortella, Amar Henni, Alberto Celotto, Marcos Jardim and Tom\'{a}s G\'{o}mez for useful discussions and suggestions. Part of
the paper was written while I was visiting the Universidade Federal
de Santa Catarina (UFSC) and the Instituto Nacional de Matem\'{a}tica
Pura e Aplicada (IMPA). I would like to thank these institutions for
their warm hospitality. This work was partially supported by INDAM-GNSAGA and by PRIN "Geometria delle variet\`{a} algebriche".

\section{Preliminaries and notation}

\subsection{Bilinear forms on sheaves}

Let $X$ be a $\mathbb{K}-$scheme with $\mathbb{K}$ an algebraically
closed field with $char(\mathbb{K})\neq2$, and let $E$ be an $\mathcal{O}_{X}-$module.
The skew-symmetric square $\Lambda^{2}E$ and the symmetric square
$S^{2}E$ of $E$ are defined as the sheafifications of the presheaves
\[
U\mapsto S^{2}E(U),\:U\mapsto\Lambda^{2}E(U).
\]
In particular, they can be naturally written as quotients of $E^{\otimes2}$.
Let $i\in Aut(E^{\otimes2})$ be the natural switch morphism (on stalks:
$i(e\otimes f)=f\otimes e$). Let $G$ be another $\mathcal{O}_{X}-$module
and $\varphi:E^{\otimes2}\rightarrow G$ a morphism. We call $\varphi$
symmetric (resp. skew-symmetric) if $\varphi\circ i=\varphi$ (resp.
$\varphi\circ i=-i$). $S^{2}E$ and $\Lambda^{2}E$ satisfy the obvious
universal properties 
\[
\xymatrix{E^{\otimes2}\ar[r]^{\,\,\,\,\,\forall symm}\ar[d] & G\\
S^{2}E\ar[ur]_{\exists!}
}
\;\xymatrix{E^{\otimes2}\ar[r]^{\,\,\,\forall skew}\ar[d] & G\\
\Lambda^{2}E\ar[ur]_{\exists!}
}
\]
and in fact fit into a split-exact sequence
\[
0\rightarrow S^{2}E\rightarrow E^{\otimes2}\rightarrow\Lambda^{2}E\rightarrow0.
\]

\begin{rem}
A bilinear form $\varphi\in Hom(E^{\otimes2},G)$ naturally corresponds
to an element of 
\[
Hom(E,G\otimes E^{\vee})\cong Hom(E^{\otimes2},G).
\]
If $\varphi$ is symmetric or skew, it corresponds to a unique form
in $Hom(S^{2}E,G)$ or $Hom(\Lambda^{2}E,G)$. We shall make a systematic
abuse of notation by calling $\varphi$ all these morphisms.
\end{rem}
The modules $S^{2}E$ and $\Lambda^{2}E$ are coherent or locally
free if $E$ is, and are well-behaved with respect to pullbacks, meaning
that for a given morphism of schemes $f:Y\rightarrow X$ one has natural
isomorphisms
\[
S^{2}f^{\star}E\cong f^{\star}S^{2}E,\:\Lambda^{2}f^{\star}E\cong f^{\star}\Lambda^{2}E.
\]
The following lemmas will be useful later.
\begin{lem}
\label{lem:KwedgeH} Let $K\rightarrow H\rightarrow E\rightarrow0$
be an exact sequence of coherent $\mathcal{O}_{X}-$modules. There
is an exact sequence
\[
K\otimes H\rightarrow\Lambda^{2}H\rightarrow\Lambda^{2}E\rightarrow0.
\]
If furthermore $H$ is locally free and $K\rightarrow H$ is injective,
define $K\wedge H$ to be the image of the subsheaf $K\otimes H\subseteq H^{\otimes2}$
under $H^{\otimes2}\rightarrow\Lambda^{2}H$. Then there is a natural
isomorphism 
\[
K\wedge H\cong K\otimes H/(K\otimes H\cap S^{2}H)
\]
and thus an exact sequence 
\[
0\rightarrow K\wedge H\rightarrow\Lambda^{2}H\rightarrow\Lambda^{2}E\rightarrow0.
\]

\end{lem}

\begin{lem}
\label{lem:gom sol cl subsch}\cite[Lemma 0.9]{GS2} Let $Y$ be a
scheme and $h:\mathcal{F}\rightarrow\mathcal{G}$ a morphism of coherent
sheaves on $X\times Y$. Assume $\mathcal{G}$ is flat over $Y$.
There exists a closed subscheme $Z\subseteq Y$ such that the following
universal property is satisfied: for any $s:S\rightarrow Y$ with
$(1_{X}\times s)^{\star}h=0$, one has a unique factorization of $s$
as $S\rightarrow Z\rightarrow Y$. 
\end{lem}

\begin{lem}
\label{lem:Spectral tor}\cite[Thm 12.1]{Mc} Let $R\rightarrow S$
be a morphism of unitary commutative rings, let $N$ be an $S-$module
and $N^{\prime}$ be an $R-$module. There exists a spectral sequence
with $E_{p,q}^{2}=Ext_{S}^{q}(Tor_{p}^{R}(S,N^{\prime}),N)$ converging
to $Ext_{R}(N^{\prime},N).$ In particular, if $N^{\prime}$ is $R-$flat,
we get an isomorphism
\[
Ext_{S}^{q}(S\otimes N^{\prime},N)\cong Ext_{R}^{q}(N^{\prime},N).
\]

\end{lem}

\begin{lem}
\label{lem:closem crit}Suppose $f:X\rightarrow Y$ is a morphism
of $\mathbb{K}-$schemes of finite type. Assume that the following
conditions hold:
\begin{enumerate}
\item the map of sets $f(\mathbb{K}):X(\mathbb{K})\rightarrow Y(\mathbb{K})$
is injective;
\item for any closed point $x\in X$ the linear map $Tf_{x}:T_{x}X\rightarrow T_{f(x)}Y$
is a monomorphism;
\item f is proper.
\end{enumerate}
Then $f$ is a closed embedding.
\end{lem}

\subsection{Framed symplectic sheaves}

Let $X$ be a projective surface, $D\subseteq X$ a big and nef divisor
and $W$ a finite-dimensional vector space.
\begin{defn}
A framed sheaf on $X$ is a pair $(E,\alpha)$ where $E$ is a torsion-free
sheaf of rank $r$ on $X$ and $\alpha:E_{D}\rightarrow\mathcal{O}_{D}\otimes W$
is an isomorphism. A morphism of framed sheaves $(E,\alpha)\rightarrow(E^{\prime},\alpha^{\prime})$
is a morphism of sheaves $f:E\rightarrow E^{\prime}$ where $\alpha^{\prime}\circ f_{D}=\lambda\circ\alpha$
for a nonzero scalar $\lambda\in\mathbb{K}$.
\end{defn}
If $(E,\alpha)$ is a framed sheaf, $E$ is locally free in a neighborhood
of $D$.

Framed sheaves are easily defined in families. An $S-$family of framed
sheaves on $X$ for a given scheme $S$ consists of a pair $(\mathcal{E},\alpha)$
where $\mathcal{E}$ is a $S-$flat sheaf on $X_{S}$ and $\alpha$
is an isomorphism 
\[
\alpha:\mathcal{E}_{D_{S}}\rightarrow D_{S}\otimes W.
\]
The pullback family via a morphism $S\rightarrow S^{\prime}$ is well
defined. We obtain a functor 
\[
\mathfrak{M}_{X}^{D}(r,n):Sch^{op}\rightarrow Set
\]
assigning to $S$ the set of isomorphism classes of $S-$families
of framed sheaves. 

Fix now some symplectic form $\Omega:\Lambda^{2}W\rightarrow\mathbb{K}$
(this forces $r=dim(W)$ to be even). This yields a symplectic form
on $\mathcal{O}_{D}\otimes W$, still denoted by $\Omega$. 
\begin{defn}
A framed symplectic sheaf is a triple $(E,a,\varphi)$ where $(E,a)$
is a framed sheaf with $det(E)\cong\mathcal{O}_{X}$ and $\varphi:\Lambda^{2}E\rightarrow\mathcal{O}_{X}$
is a morphism satisfying 
\begin{equation}\label{eq:compatib}
\varphi_{D}=\Omega\circ\Lambda^{2}a.
\end{equation}
\end{defn}

\begin{rem}The morphism $\varphi:E\rightarrow E^{\vee}$ is an isomorphism once
restricted to $D$; consequently, the same holds on an open neighborhood
of $D$. As $det(E)\cong det(E^{\vee})\cong\mathcal{O}_{X}$, we obtain $c_{1}(E)=0$. Furthermore,
we deduce that $det(\varphi)$ is a nonzero constant. It follows that $\varphi$
induces an isomorphism $E_{U}\rightarrow E_{U}^{\vee}$ on the entire
open $2-$codimensional set $X\backslash sing(E)=U\supset D$. \\
If $(E,a)$ is a framed sheaf with $c_{1}(E)=0$ and we are given a morphism $\varphi:\Lambda^{2}E\rightarrow\mathcal{O}_{X}$ satisfying the compatibility condition \ref{eq:compatib}, 
we get in fact $det(E)\cong\mathcal{O}_{X}$. Indeed, since $\varphi$ is an isomorphism in a neighborhood of $D$, we obtain an exact sequence
\[
0\rightarrow E\rightarrow E^{\vee}\rightarrow coker(\varphi)\rightarrow0.
\]
The sheaf  $coker(\varphi)$ must be supported on a subscheme of dimension strictly smaller than 2, and $c_{1}(coker(\varphi))=0$ forces
\[
dim(coker(\varphi))=0.
\]
This implies that $\varphi$ must be an isomorphism outside a zero-dimensional subscheme of $X$, from which 
\[
det(E)\cong det(E^{\vee})
\]
follows. Finally, consider the dual map
\[
\varphi^{\vee}:E^{\vee \vee}\rightarrow E^{\vee};
\]
it is a skew-symmetric isomorphism of vector bundles, i.e. a symplectic form. We may conclude $det(E^\vee)\cong\mathcal{O}_{X}$.
From now on, we will always omit the hypothesis on the determinant while working with framed symplectic sheaves, as we will consider exclusively framed sheaves whose first Chern class vanishes.  Thus, the discrete invariants of framed symplectic sheaves will be the positive integer $r$
and the nonnegative integer $c_{2}$.
\end{rem}

We define a morphism of framed symplectic sheaves $(E,\alpha,\varphi)\rightarrow(E^{\prime},\alpha^{\prime},\varphi^{\prime})$
to be a morphism $f$ of framed sheaves where $\varphi^{\prime}\circ\Lambda^{2}f=\lambda\varphi$
for a (again nonzero) $\lambda\in\mathbb{K}$.
\begin{rem}
\label{rem:rigidity}We shall soon see (Prop. \ref{prop:Bound+uniq})
that, for a given pair of framed sheaves $(E,a)$ and $(F,b)$ with the same invariants, there
is at most one isomorphism $f:E\rightarrow F$ satisfying 
\[
a=b\circ f_{D}.
\]
This implies that $(E,a)$ can support \emph{at most one }structure
of framed symplectic sheaf, in the following sense. If $\varphi:\Lambda^{2}E\rightarrow\mathcal{O}_{X}$
is a symplectic form, it induces an isomorphism $\varphi^{\vee}:E^{\vee\vee}\rightarrow E^{\vee}$. Furthermore, $E^{\vee\vee}$ and $E^{\vee}$ inherit
framings from $(E,a)$, and $\varphi^{\vee}$ preserves these framings,
as $\varphi$ is $\Omega-$compatible with $a$. As a consequence,
any other symplectic form $\varphi^{\prime}$ on $E$ satisfies $(\varphi^{\prime})^{\vee}=\varphi^{\vee}$,
from which $\varphi=\varphi^{\prime}$ follows. In particular:\end{rem}
\begin{lem}
If $(E,\alpha,\varphi)$ is a framed symplectic sheaf, $Hom(\Lambda^{2}E,\mathcal{O}_{X}(-D))=0$
(i.e., the symplectic form has no nontrivial infinitesimal automorphisms).\end{lem}
\begin{proof}
Let $\psi:\Lambda^{2}E\rightarrow\mathcal{O}_{X}(-D)$ be a morphism.
By means of the exact sequence 
\[
0\rightarrow\mathcal{O}_{X}(-D)\rightarrow\mathcal{O}_{X}\rightarrow\mathcal{O}_{D}\rightarrow0
\]
we can think of $\psi$ as a morphism $\Lambda^{2}E\rightarrow\mathcal{O}_{X}$
which vanishes on $D$. Now, for a nonzero scalar $\lambda$ consider
$\psi_{\lambda}=\psi+\lambda\varphi$. If we choose a square root
of $\lambda^{1/2}$, we obtain that $(E,\lambda^{1/2}\alpha,\psi_{\lambda})$
is a framed symplectic sheaf, but also $(E,\lambda^{1/2}\alpha,\lambda\varphi)$
is. This forces $\psi_{\lambda}=\lambda\varphi$, i.e. $\psi=0$.
\end{proof}
We can define framed symplectic sheaves in families again; an $S-$family
of framed symplectic sheaves will be a triple $(\mathcal{E},\alpha,\Phi)$
with $(\mathcal{E},\alpha)$ an $S-$family of framed sheaves, and
$\Phi:\Lambda^{2}\mathcal{E}\rightarrow\mathcal{O}_{X_{S}}$ a morphism
such that $\Phi\mid_{D_{S}}=\Omega\circ\Lambda^{2}\alpha.$ The corresponding
functor will be denoted $\mathfrak{M}_{X,\Omega}^{D}(r,n)$.

\section{Moduli spaces of framed symplectic sheaves on surfaces}

\subsection{Framed sheaves as Huybrechts-Lehn framed pairs}

It is possible to construct a fine moduli space for the functor $\mathfrak{M}_{X}^{D}(r,n)$.
Let $F$ be a fixed sheaf on $X$. A framed module is a pair $(E,a)$
where $E$ is a coherent sheaf on $X$, and $a:E\rightarrow F$ is
a morphism. A framed sheaf $(E,a)$ is a special example of framed
module, with $F=\mathcal{O}_{D}\otimes W$, $E$ torsion-free and
$a$ inducing an isomorphism once restricted to $D$. In \cite[Def. 1.1 and Thm 2.1]{HL},
a (semi)stability condition depending on a numerical polynomial $\delta$
and on a fixed polarization $H$ is defined, and a boundedness result
is provided for framed modules. 

Let $c\in H^{\star}(X,\mathbb{Q})$. In \cite[Thm 3.1]{BM}, it is
shown that there exist a polarization $H$ and a numerical polynomial
$\delta$ such that any framed sheaf $(E,a)$ with Chern character
$c(E)=c$ is $\delta$-stable as a framed module. This is a crucial
step in order to realize the moduli space of framed sheaves $\mathcal{M}_{X}^{D}(r,n)$
as an open subscheme of the moduli space of $\delta$-semistable framed
modules as defined in \cite{HL}. 

On the other hand, in \cite{GS1} the authors present the construction
of a coarse moduli space for semistable symplectic sheaves. A symplectic
sheaf is a pair $(E,\varphi)$ where $E$ is a coherent torsion-free
sheaf on $X$ and $\varphi:\Lambda^{2}E\rightarrow\mathcal{O}_{X}$
is a morphism inducing a symplectic form on the maximal open subset
of $X$ over which $E$ is locally free. We will construct a fine
moduli space for $\mathfrak{M}_{X,\Omega}^{D}(r,n)$ by means of a
blend of the previous two constructions. Fix two integers $r>0$ and
$n\geq0$, let $c\in H^{\star}(X,\mathbb{Q})$ be the Chern character
of a sheaf $E\in Coh(X)$ with generic rank $r$, $c_{1}(E)=0$ (as
we shall be working with symplectic sheaves) and $c_{2}(E)=n$, and
fix $H$ a polarization as in \cite[Thm 3.1]{BM}. Let $P_{r,n}=P_{E}$
be the corresponding Hilbert polynomial. Keeping now in mind that
in this setting framed sheaves are a particular type of \emph{stable}
framed modules, we obtain the following proposition as an immediate
consequence of \cite[Thm. 2.1]{HL} and \cite[Lemma 1.6]{HL}.
\begin{prop}
\label{prop:Bound+uniq}Let $P=P_{r,n}$. 
\begin{enumerate}
\item There exists a positive integer $m_{0}$ such that for any $m\geq m_{0}$
and for any framed sheaf $(E,a)$ on $X$ with Hilbert polynomial
$P$, $E$ is $m-$regular, $H^{i}(X,\mathcal{O}_{X}(m))=0\:\forall i>0,$
$H^{1}(X,\mathcal{O}_{D}(m))=0$ and $P(m)=h^{0}(E(m))$. 
\item Any morphism of framed sheaves $(E,a),\,(F,b)$ with the same Hilbert
polynomial is an isomorphism. Furthermore, there exists a unique morphism
$f:(E,a)\rightarrow(F,b)$ satisfying $f{}_{D}\circ b=a$, and any
other morphism between them is a nonzero multiple of $f$.
\end{enumerate}
\end{prop}

\subsection{Parameter spaces}

We need to recall how to endow the set of isomorphism classes of framed
sheaves $\mathcal{M}_{X}^{D}(r,n)$ with a scheme structure, which
makes it a fine moduli space for the functor $\mathfrak{M}_{X,\Omega}^{D}(r,n)$
(for the full details, see \cite{HL}), as it will be useful for constructing
the moduli space $\mathcal{M}_{X,\Omega}^{D}(r,n)$. The construction
is standard; we obtain both moduli spaces as geometric quotients of
suitable parameter spaces, defined using $Quot$ schemes. The relevant
parameter spaces will be defined in this subsection.

Fix a polynomial $P=P_{r,n}$ and a positive integer $m\gg0$ as in
Prop. \ref{prop:Bound+uniq}, and let $V$ be a vector space of dimension
$P(m)$ . Let $H=V\otimes\mathcal{O}_{X}(-m)$. Consider the projective
scheme
\[
Hilb(H,P)\times\mathbb{P}(Hom(V,H^{0}(\mathcal{O}_{D}(m))\otimes W)^{\vee})=:Hilb\times P_{fr},
\]
where $Hilb(H,P)$ is the Grothendieck $Quot$ scheme parameterizing
equivalence classes of quotients 
\[
q:H\rightarrow E,\,P_{E}=P
\]
 on $X$. Define $Z$ as the subset of pairs $([q:H\rightarrow E],A)$
such that the map $H\rightarrow\mathcal{O}_{D}\otimes W$ induced
by $A$ factors through $E$:
\[
\xymatrix{H\ar[d]_{q}\ar[r]^{A} & \mathcal{O}_{D}\otimes W\\
E\ar[ur]_{\bar{A}}
}
\]

$Z$ can be in fact interpreted as a closed subscheme as follows.
Choose a universal quotient 
\[
q_{Hilb}:V\otimes\mathcal{O}_{X\times Hilb}\otimes p_{X}^{\star}\mathcal{O}_{X}(-m)=:\mathcal{H}\twoheadrightarrow\mathcal{E}.
\]
The pullback of the universal map 
\[
V\otimes\mathcal{O}_{P_{fr}}\rightarrow H^{0}(\mathcal{O}_{D}(m)\otimes W)\otimes\mathcal{O}_{P_{fr}}
\]
to $X\times Hilb\times P_{fr}$ yields a morphism
\[
\mathcal{A}_{univ}:V\otimes\mathcal{O}_{X\times Hilb\times P_{fr}}\otimes p_{X}^{\star}\mathcal{O}_{X}(-m)\rightarrow\mathcal{O}_{D\times Hilb\times P_{fr}}\otimes W.
\]
The pullback of $q_{Hilb}$, induces a diagram 
\[
\xymatrix{0\ar[r] & \mathcal{K}\ar[r]\ar[dr]_{\mathcal{A}_{univ}^{\mathcal{K}}} & V\otimes\mathcal{O}_{X\times Hilb\times P_{fr}}\otimes p_{X}^{\star}\mathcal{O}_{X}(-m)\ar[r]\ar[d]^{\mathcal{A}_{univ}} & p_{Hilb\times X}^{\star}(\mathcal{E})\ar[r]\ar@{-->}[dl] & 0\\
 &  & \mathcal{O}_{D\times Hilb\times P_{fr}}\otimes W
}
\]
The closed subscheme $Z\subseteq Hilb\times P_{fr}$ we are looking
for is the one defined by Lemma \ref{lem:gom sol cl subsch}, where
$Y=Hilb\times P_{fr}$ and $h=\mathcal{A}_{univ}^{\mathcal{K}}$.
We denote by $\overset{\circ}{Z}\subseteq Z$ the open subscheme defined
by requiring that the pullback $a=\bar{A}_{D}:E_{D}\rightarrow\mathcal{O}_{D}\otimes W$
is a framing. 

Similarly, we can take the product 
\[
Hilb(H,P)\times\mathbb{P}(Hom(\Lambda^{2}V\rightarrow H^{0}(\mathcal{O}_{X}(2m)))^{\vee})=:Hilb\times P_{symp}
\]
and consider the closed subscheme $Z^{\prime}\subseteq Hilb\times P_{symp}$
given by pairs $([q],\phi)$ such that the map $\Lambda^{2}H\rightarrow\mathcal{O}_{X}$
induced by $\phi$ descends to some $\varphi:\Lambda^{2}E\rightarrow\mathcal{O}_{X}$:
\[
\xymatrix{\Lambda^{2}H\ar[d]_{\Lambda^{2}q}\ar[r]^{\phi} & \mathcal{O}_{X}\\
\Lambda^{2}E\ar[ur]_{\varphi}
}
\]
 We define the very last closed subscheme $Z_{\Omega}\subseteq Hilb\times P_{fr}\times P_{symp}$
as follows. First, consider the scheme-theoretic intersection 
\[
p_{Hilb\times P_{fr}}^{-1}(Z)\cap p_{Hilb\times P_{symp}}^{-1}(Z^{\prime}),
\]
whose closed points are triples $([q],A,\phi)$ satisfying: $A$ descends
to $\alpha:E\rightarrow\mathcal{O}_{D}$, $\phi$ descends to $\varphi:\Lambda^{2}E\rightarrow\mathcal{O}_{X}$.
This scheme has again a closed subscheme defined by triples satisfying
the following compatibility on $D$:
\[
\varphi_{D}=\Omega\circ\Lambda^{2}\alpha_{D}:\Lambda^{2}E_{D}\rightarrow\mathcal{O}_{D}.
\]
Call this subscheme $Z_{\Omega}$. We denote by $\overset{\circ}{Z}_{\Omega}$
the preimage of $\overset{\circ}{Z}\subseteq Z$ under the projection
$\pi:Z_{\Omega}\rightarrow Z$.
\begin{thm}
\label{thm: clemb paramsp}The restriction $\overset{\circ}{\pi}:\overset{\circ}{Z}_{\Omega}\rightarrow\overset{\circ}{Z}$
is a closed embedding. \end{thm}
\begin{rem}
\label{rem:pi pallino}The schemes we are dealing with are in fact
of finite type; this means that we can apply Lemma \ref{lem:closem crit}
to prove the theorem. The morphism $\overset{\circ}{\pi}$ is clearly
proper, as a base change of a map between projective schemes. For
injectivity, it is enough to note that there exists only one structure
of symplectic sheaf on a given framed sheaf, in the sense of Remk.
\ref{rem:rigidity}. It follows that we only need to prove the claim
on tangent spaces: for any triple $([q],A,\phi)$ the tangent map
$T_{([q],A,\phi)}\overset{\circ}{Z}_{\Omega}\rightarrow T_{([q],A)}\overset{\circ}{Z}$
is injective. This motivates the following infinitesimal study for
the parameter spaces.
\end{rem}

\subsection{\label{sub:infparam}Infinitesimal study: parameter spaces}

Let $A$ be an artinial local $\mathbb{K}-$algebra. We define the
sheaves on $X_{A}=X\times Spec(A)$ 
\[
\mathcal{H}:=H\otimes\mathcal{O}_{A}=V\otimes O_{X_{A}}(-m);
\]
\[
\mathcal{D}=\mathcal{O}_{D}\otimes\mathcal{O}_{A}\otimes W.
\]
 The aim of the present subsection is to compute the tangent spaces
of the relative versions $Z^{A}$ and $Z_{\Omega}^{A}$ of the previously
defined parameter spaces. In other words, we have 
\[
Z^{A}\subseteq Quot_{X_{A}}(\mathcal{H},P)\times\mathbb{P}(Hom(\mathcal{H},\mathcal{D})^{\vee})
\]
defined as the closed subscheme representing the functor assigning
to an $A-$scheme $T$ the set
\[
\{q_{T}:V\otimes\mathcal{O}_{X_{T}}(-m)\twoheadrightarrow\tilde{\mathcal{E}},\alpha_{T}:\tilde{\mathcal{E}}\rightarrow\mathcal{D}_{T}\mid\tilde{\mathcal{E}}\,\,T-\mbox{flat},\,P_{\tilde{\mathcal{E}}}=P\},
\]
and 

\[
Z_{\Omega}^{A}\subseteq Z^{A}\times\mathbb{P}(Hom(\Lambda^{2}\mathcal{H},\mathcal{O}_{X_{A}})^{\vee})
\]
representing the functor assigning to an $A-$scheme $T$ the set
\[
\{(q_{T},\alpha_{T})\in Z^{A}(T),\varphi_{T}:\Lambda^{2}\tilde{\mathcal{E}}\rightarrow\mathcal{O}_{X_{T}}\mid\varphi_{T}\mid_{D\times T}=\Omega\circ\Lambda^{2}\alpha_{T}\mid_{D\times T}\}
\]
The computation of the tangent spaces to $Z^{A}$ was already performed
in \cite{HL}, and goes as follows.

Let $q:\mathcal{H}\rightarrow\mathcal{E}$ be a quotient with $P_{\mathcal{E}}=P$.
Let $ker(q):=\mathcal{K}\overset{\iota}{\hookrightarrow}\mathcal{H}$.
The tangent space to $Quot_{X_{A}}(\mathcal{H},P)$ at the point $[q:\mathcal{H}\rightarrow\mathcal{E}]$
is naturally isomorphic to the vector space $Hom_{X_{A}}(\mathcal{K},\mathcal{E})$.
Indeed, writing $S=Spec(A[\varepsilon])$, and using the universal
property of $Quot$, the tangent space may be indeed identified with
the set of equivalence classes of quotients $\tilde{q}:\mathcal{H}_{S}\rightarrow\tilde{\mathcal{E}}$
on $X_{S}$ (where $\mathcal{H}_{S}=\mathcal{H}\otimes\mathbb{C}[\varepsilon]=H\otimes A[\varepsilon]$
, and $\tilde{\mathcal{E}}$ an $S-$flat sheaf on $X_{S}$) reducing
to $q$ mod $\varepsilon$. Write $q_{S}:\mathcal{H}_{S}\rightarrow\mathcal{E}_{S}$
for the pullback of $q$ to $X_{S}$. For a given $\tilde{q}$ with
$ker(\tilde{q})=\tilde{\mathcal{K}}$, the map $q_{S}\mid_{\mathcal{\tilde{\mathcal{K}}}}$
takes values in $\varepsilon\cdot\mathcal{E}_{S}$ and factors throug
$\mathcal{K}$, defining a morphism $\mathcal{K}\rightarrow\varepsilon\cdot\mathcal{E}_{S}\cong\mathcal{E}$.
Viceversa, a given map $\gamma:\mathcal{K}\rightarrow\mathcal{E}$
defines 
\[
\tilde{\mathcal{K}}\subseteq\mathcal{H}_{S},\,\tilde{\mathcal{K}}=\rho^{-1}(\mathcal{K})\cap ker(q_{S}+\gamma\circ\rho),
\]
where $\rho:\mathcal{H}_{S}\twoheadrightarrow\mathcal{H}$ is the
natural projection induced by $A[\varepsilon]\twoheadrightarrow A$.
The tangent space to $\mathbb{P}(Hom(\mathcal{H},\mathcal{D})^{\vee})$
at a point $[\mathcal{A}]$ and the tangent space to $\mathbb{P}(Hom(\Lambda^{2}\mathcal{H},\mathcal{O}_{A})^{\vee})$
at a point $[\Phi]$ are identified respectively with the quotients
\[
Hom(\mathcal{H},\mathcal{D})/A\cdot\mathcal{A};\:\:\:Hom(\Lambda^{2}\mathcal{H},\mathcal{O}_{A})/A\cdot\Phi.
\]
To see why this is the case, we apply again universal properties.
An element of 
\[
\mathbb{P}(Hom(\mathcal{H},\mathcal{D})^{\vee})(S\rightarrow Spec(A)),\,Spec(\mathbb{K})\mapsto[\mathcal{A}]\in\mathbb{P}(Hom(\mathcal{H},\mathcal{D})^{\vee})
\]
is the same thing as a morphism $\tilde{\mathcal{A}}:\mathcal{H}_{S}\rightarrow\mathcal{D}_{S}$
(up to units in $A[\varepsilon]$), reducing to $\mathcal{A}:\mathcal{H}\rightarrow\mathcal{D}$
mod $\varepsilon$ (up to units in $A$). Such a morphism may be represented
as
\[
h+\varepsilon h^{\prime}\mapsto\mathcal{A}(h)+\varepsilon(\mathcal{A}(h^{\prime})+\mathcal{B}(h))
\]
with $\mathcal{B}:\mathcal{H}\rightarrow\mathcal{D}$ an $\mathcal{O}_{X_{A}}-$linear
map. The units in $A[\varepsilon]$ preserving this map modulo $\varepsilon$
are exactly those of the form $1+\lambda\varepsilon$ with $\lambda\in A$,
and they send $\mathcal{B}$ to $\mathcal{B}+(\lambda-1)\mathcal{A}$.
This proves the claim. 

We proceed in an analogous way in the other case. An element of 
\[
\mathbb{P}(Hom(\Lambda^{2}\mathcal{H},\mathcal{O}_{A})^{\vee})(S\rightarrow Spec(A)),\,Spec(\mathbb{K})\mapsto[\Phi]\in\mathbb{P}(Hom(\Lambda^{2}\mathcal{H},\mathcal{O}_{A})^{\vee})
\]
 is the same thing as a morphism $\tilde{\Phi}:\Lambda^{2}\mathcal{H}_{S}\rightarrow\mathcal{O}_{S}$
(up to units in $A[\varepsilon]$), reducing to $\Phi$ modulo $\varepsilon$
(up to units in $A$). Again, one can choose a representative
\[
(h+\varepsilon h^{\prime})\otimes(g+\varepsilon g^{\prime})=\Phi(h\otimes g)+\varepsilon(\Phi(h^{\prime}\otimes g)+\Phi(h\otimes g^{\prime})+\psi(h\otimes g))
\]
with $\psi:\Lambda^{2}\mathcal{H}\rightarrow\mathcal{O}_{X_{A}}$,
and two different maps $\psi$ will induce the same morphism up to
units if and only if they differ by an $A-$multiple of $\Phi$. 

Let $(q,\alpha)\in Z^{A}$ with $\alpha\circ q=\mathcal{A}$. The
pairs 
\[
(\gamma,\mathcal{B})\in Hom(\mathcal{K},\mathcal{E})\oplus Hom(\mathcal{H},\mathcal{D})/A\cdot\mathcal{A}=T_{(q,\mathcal{A})}(Quot(\mathcal{H},P)\times\mathbb{P}(Hom(\mathcal{H},\mathcal{D})^{\vee}))
\]
belonging to $T_{(q,\alpha)}Z^{A}$ are characterized by the equation
\[
\mathcal{B}\circ\iota=\alpha\circ\gamma,
\]
see \cite{HL}.

Now, let $(\tilde{q},\tilde{A},\tilde{\Phi})\in Z_{H}^{A}(S)$. This
means that the morphism $\tilde{\Phi}:\Lambda^{2}\mathcal{H}_{S}\rightarrow\mathcal{O}_{S}$
descends to some $\tilde{\varphi}:\Lambda^{2}\mathcal{E}_{S}\rightarrow\mathcal{O}_{S}$
via $\tilde{q}$, i.e. its restriction to $\tilde{\mathcal{K}}\otimes\mathcal{H}_{S}$
vanishes. In local sections, an element $h+\varepsilon h^{\prime}\in\mathcal{H}_{S}$
belongs to $\tilde{\mathcal{K}}$ if and only if $q(h)=0$ and $\gamma(h)=-q(h^{\prime})$.
We obtain:
\[
\tilde{\Phi}((h+\varepsilon h^{\prime})\otimes(g+\varepsilon g^{\prime}))=\varepsilon(\Phi(h^{\prime}\otimes g)+\psi(h\otimes g))=\varepsilon(-\varphi(\gamma(h)\otimes q(g))+\psi(h\otimes g)).
\]
This quantity vanishes if and only if the equation 
\[
\psi(\iota\otimes1)=\varphi(\gamma\otimes q)
\]
holds.

The triple is required to satisfy another condition, namely the compatibility
on the divisor $D_{S}$:
\[
(\tilde{\varphi})\mid_{D_{S}}=\Omega\circ(\tilde{\mathcal{A}}^{\otimes2})\mid_{D_{S}}.
\]
We make an abuse of notation by writing $h+\varepsilon h^{\prime}$
for sections of $\mathcal{H}_{S}\mid_{D_{S}}\cong V\otimes\mathcal{O}_{D_{S}}(-m)$;
we get 
\[
\tilde{\Phi}((h+\varepsilon h^{\prime})\otimes(g+\varepsilon g^{\prime}))=\Omega((\mathcal{A}(h)+\varepsilon(\mathcal{A}(h^{\prime})+\mathcal{B}(h)))\otimes(\mathcal{A}(g)+\varepsilon(\mathcal{A}(g^{\prime})+\mathcal{B}(g))))
\]

\[
\Phi(h\otimes g)+\varepsilon(\Phi(h^{\prime}\otimes g)+\Phi(h\otimes g^{\prime})+\psi(h\otimes g))=\Omega(\mathcal{A}(h)\otimes\mathcal{A}(g))+
\]
\[
+\varepsilon(\Omega(\mathcal{A}(h)\otimes\mathcal{A}(g^{\prime}))+\Omega(\mathcal{A}(h)\otimes\mathcal{B}(g))+\Omega(\mathcal{A}(h^{\prime})\otimes\mathcal{A}(g))+\Omega(\mathcal{B}(h)\otimes\mathcal{A}(g)))
\]
Since 
\[
(\varphi)\mid_{D_{S}}=\Omega\circ(\mathcal{A}^{\otimes2})\mid_{D_{S}},\,\alpha\circ q=\mathcal{A}
\]
holds by hypothesis, we can simplify:
\[
\psi(h\otimes g)=\Omega(\alpha\circ q(h)\otimes\mathcal{B}(g))+\Omega(\mathcal{B}(h)\otimes\alpha\circ q(g))).
\]
We have obtained the description of the tangent space we needed.
\begin{prop}
There are isomorphisms: 
\[
T_{(q,\mathcal{A})}Z^{A}\cong\{(\gamma,\mathcal{B})\in Hom(\mathcal{K},\mathcal{E})\oplus Hom(\mathcal{H},\mathcal{D})/A\cdot\mathcal{A}\mid\mathcal{B}\circ\iota=\alpha\circ\gamma\};
\]
\[
T_{(q,\mathcal{A},\Phi)}Z_{\Omega}^{A}\cong\{((\gamma,\mathcal{B}),\psi)\in T_{(q,\mathcal{A})}Z^{A}\oplus Hom(\Lambda^{2}\mathcal{H},\mathcal{O}_{A})/A\cdot\Phi\mid\psi(\iota\otimes1)=\varphi(\gamma\otimes q),
\]
\[
\psi\mid_{D_{A}}=\Omega(\mathcal{A}\otimes\mathcal{B}+\mathcal{B}\otimes\mathcal{A})\}
\]

\end{prop}

\subsection{Moduli spaces}

We can now apply the results in the previous subsection to prove Thm.
\ref{thm: clemb paramsp}.
\begin{proof}
For a quotient $q:H\rightarrow E$, we denote by $\iota:K\rightarrow H$
its kernel. We know that the tangent space $T_{([q],A)}Z$, which
is naturally a subspace of $Hom(K,E)\oplus(Hom(H,D_{W})/\mathbb{C}A)$,
can be described as the subset of pairs $(\gamma,[B])$ satisfying
the equation 
\[
\bar{A}\circ\gamma=B\mid_{K}.
\]
The space $T_{([q],A,\phi)}\subseteq T_{([q],A)}Z\oplus(Hom(\Lambda^{2}H,\mathcal{O}_{X})/\mathbb{K}\phi$
can be instead identified with the subspace of triples $(\gamma,[B],[\psi])$
defined by the equations 
\[
\psi(\iota\otimes1_{H})=\varphi(\gamma\otimes q);\;\psi_{D}=\Omega(A\otimes B+B\otimes A).
\]
The differential map $T_{([q],A,\phi)}\overset{\circ}{Z}_{\Omega}\rightarrow T_{([q],A)}\overset{\circ}{Z}$
is just the projection 
\[
(\gamma,[B],[\psi])\mapsto(\gamma,[B]).
\]
Now, suppose an element of type $(0,\lambda\mathcal{A},[\psi])$ belongs
to $T_{([q],A,\phi)}\overset{\circ}{Z}_{\Omega}$. This means 
\[
\psi(\iota\otimes1_{\Omega})=0,
\]
so that $\psi$ descends to some $\bar{\psi}\in Hom(\Lambda^{2}E,\mathcal{O}_{X})$.
Furthermore, we get 
\[
\psi_{D}=\lambda^{2}\Omega(\mathcal{A}^{\otimes2}+\mathcal{A}^{\otimes2}\circ i)=0\implies\psi_{D}\in Hom(\Lambda^{2}E,\mathcal{O}_{X}(-D))=0.
\]
We have proved injectivity for the tangent map, and this finishes
the proof thanks to Remk. \ref{rem:pi pallino}.
\end{proof}
The natural action of the group $SL(V)$ on the bundle $H=V\otimes\mathcal{O}_{X}(-m)$
induces $SL(V)$ actions on the schemes $Z$ and $Z_{\Omega}$, and
the map $Z_{\Omega}\rightarrow Z$ is equivariant. In addition, the
open subschemes $\overset{\circ}{Z}$ and $\overset{\circ}{Z}_{\Omega}$
are invariant. As already announced, the following theorem holds:
\begin{thm}
The $SL(V)-$scheme $\overset{\circ}{Z}$ admits a geometric quotient
$\overset{\circ}{Z}/SL(V)$; this scheme is a fine moduli space for
the functor $\mathfrak{M}_{X}^{D}(r,n)$, and will be denoted $\mathcal{M}_{X}^{D}(r,n)$.\end{thm}
\begin{defn}
We define the scheme $\mathcal{M}_{X,\Omega}^{D}(r,n)$ to be the
closed subscheme 
\[
\overset{\circ}{Z}_{\Omega}/SL(V)\subseteq\mathcal{M}_{X}^{D}(r,n).
\]
 \end{defn}
\begin{thm}
\textup{The scheme $\mathcal{M}_{X,\Omega}^{D}(r,n)$ is a fine moduli
space for the functor $\mathfrak{M}_{X}^{D}(r,n)$.}\end{thm}
\begin{proof}
Let $S$ be any scheme of finite type and let $(\mathcal{E}_{S},a_{S},\varphi_{S})$
be an $S-$family of framed symplectic sheaves. Consider the sheaf
$\mathcal{V}_{S}:=p_{S\star}(\mathcal{E}_{S}\otimes p_{X}^{\star}\mathcal{O}_{X}(m))$.
Since for any $s\in S(\mathbb{K})$ we have 
\[
H^{i}(\mathcal{E}_{S}\otimes p_{X}^{\star}\mathcal{O}_{X}(m)\mid_{\{s\}\times X})=0\,\,\forall i>0
\]
 by Lemma \ref{prop:Bound+uniq}, $\mathcal{V}_{S}$ is locally of
rank $P(m)$. Furthermore, the natural map 
\[
q_{S}:p_{S}^{\star}\mathcal{V}_{S}\rightarrow\mathcal{E}_{S}\otimes p_{X}^{\star}\mathcal{O}_{X}(m)
\]
is surjective, as it restricts to 
\[
H^{0}(\mathcal{E}_{S}\otimes p_{X}^{\star}\mathcal{O}_{X}(m)\mid_{\{s\}\times X})\otimes\mathcal{O}_{X}\rightarrow\mathcal{E}_{S}\otimes p_{X}^{\star}\mathcal{O}_{X}(m)\mid{}_{\{s\}\times X}
\]
on fibres, and we may apply again Lemma \ref{prop:Bound+uniq}. From
$\varphi_{S}$ we obtain a map 
\[
\mathcal{V}_{S}^{\otimes2}\rightarrow p_{S\star}(\mathcal{O}_{X_{S}}\otimes p_{X}^{\star}\mathcal{O}_{X}(2m))=\mathcal{O}_{S}\otimes H^{0}(X,\mathcal{O}_{X}(2m)).
\]
Indeed, since the higher $p_{S}-$pushforwards of $\mathcal{E}_{S}\otimes p_{X}^{\star}\mathcal{O}_{X}(m)$
vanish, the formula
\[
Rp_{T\star}(\mathcal{E}_{S}\otimes p_{X}^{\star}\mathcal{O}_{X}(m)\overset{L}{\otimes}\mathcal{E}_{S}\otimes p_{X}^{\star}\mathcal{O}_{X}(m))\cong Rp_{S\star}(\mathcal{E}_{S}\otimes p_{X}^{\star}\mathcal{O}_{X}(m))\overset{L}{\otimes}Rp_{S\star}(\mathcal{E}_{S}\otimes p_{X}^{\star}\mathcal{O}_{X}(m))\cong\mathcal{V}_{S}^{\otimes2}
\]
holds, and induces the desired map. It still is skew-symmetric, an
thus yields 
\[
p_{S\star}(\varphi_{S}\otimes p_{X}^{\star}\mathcal{O}_{X}(2m)):\Lambda^{2}\mathcal{V}_{S}\rightarrow\mathcal{O}_{S}\otimes H^{0}(X,\mathcal{O}_{X}(2m)).
\]
We define the morphism $\phi_{S}$ to be the composition
\[
\phi_{S}:\Lambda^{2}p_{X}^{\star}\mathcal{V}_{S}\rightarrow\mathcal{O}_{X_{S}}\otimes H^{0}(X,\mathcal{O}_{X}(2m))\rightarrow\mathcal{O}_{X_{S}}\otimes p_{X}^{\star}\mathcal{O}_{X}(2m).
\]
Also $a_{S}$ similarly induces a map 
\[
\mathcal{V}_{S}\rightarrow\mathcal{O}_{S}\otimes H^{0}(X,\mathcal{O}_{D}(m)\otimes W)
\]
and we define as above 
\[
A_{S}:p_{X}^{\star}\mathcal{V}_{S}\rightarrow\mathcal{O}_{D_{S}}\otimes p_{X}^{\star}\mathcal{O}_{X}(m)\otimes W.
\]
By construction, the diagrams
\[
\xymatrix{\Lambda^{2}p_{X}^{\star}\mathcal{V}_{S}\ar[d]_{\Lambda^{2}q_{s}}\ar[r]^{\!\!\!\!\phi_{S}} & \mathcal{O}_{X_{S}}\otimes p_{X}^{\star}\mathcal{O}_{X}(m)\\
\Lambda^{2}\mathcal{E}_{S}\otimes p_{X}^{\star}\mathcal{O}_{X}(2m)\ar[ur]_{\varphi_{S}(2m)}
}
\:\xymatrix{p_{X}^{\star}\mathcal{V}_{S}\ar[d]_{q_{s}}\ar[r]^{\!\!\!\!\!\!\!\!\!\!\!\!\!A_{S}} & \mathcal{O}_{D_{S}}\otimes p_{X}^{\star}\mathcal{O}_{X}(m)\otimes W\\
\mathcal{E}_{S}\otimes p_{X}^{\star}\mathcal{O}_{X}(m)\ar[ur]_{a_{S}(m)}
}
\]
are commutative. 

Define an open covering $S=\bigcup S_{i}$ such that $\mathcal{V}_{S}$
trivializes over each $S_{i}$, and fix isomorphisms $\mathcal{O}_{S_{i}}\otimes V\cong\mathcal{V}_{S_{i}}$,
where $V$ is a vector space of a dimension $P(m)$; the trivializations
differ on the overlaps $S_{ij}$ by a map $S_{ij}\rightarrow GL(V)$.
Restricting the maps $q_{S}$, $\phi_{S}$ and $A_{S}$ to $S_{i}\times X$
and twisting by $p_{X}^{\star}\mathcal{O}_{X}(-m)$ we obtain maps
$S_{i}\rightarrow\overset{\circ}{Z}_{\Omega}$, which glue to a map
$S\rightarrow\overset{\circ}{Z}_{\Omega}/SL(V)=\mathcal{M}_{X,\Omega}^{D}(r,n)$.
We note that acting via $SL(V)$ or $GL(V)$ does not make a real
difference as the natural action by $\mathbb{G}_{m}$ on the parameter
spaces is trivial. 

The resulting natural transformation 
\[
\mathfrak{M}_{X,\Omega}^{D}(r,n)\rightarrow\mathcal{M}_{X,\Omega}^{D}(r,n)
\]
makes $\mathcal{M}_{X,\Omega}^{D}(r,n)$ into a coarse moduli space
for framed symplectic sheaves; indeed, since $\overset{\circ}{Z}_{\Omega}$
parameterizes a tautological family of framed symplectic sheaves,
for any scheme $N$ and for any natural transformation $\mathfrak{M}_{X,\Omega}^{D}(r,n)\rightarrow Hom(\_,N)$
we obtain a map $\overset{\circ}{Z}_{\Omega}\rightarrow N$. This
map has to be $SL(V)$ invariant as two points of $\overset{\circ}{Z}_{\Omega}$
that lie in the same orbit define isomorphic framed sheaves. The fact
that the moduli space is indeed fine can be proved by noting that
framed symplectic sheaves are rigid, i.e. by applying Remk. \ref{rem:rigidity}
and proceeding as in \cite[proof of Main Theorem]{HL}.\end{proof}
\begin{rem}
The moduli spaces $\mathcal{M}_{X}^{D}(r,n)$ and $\mathcal{M}_{X,\Omega}^{D}(r,n)$
both contain open subschemes of isomorphism classes of locally free
sheaves. These are fine moduli spaces fo framed $SL_{r}$ and $SP_{r}$
principal bundles, and will be respectively denoted $\mathcal{M}_{X}^{D,reg}(r,n)$
and $\mathcal{M}_{X,\Omega}^{D,reg}(r,n)$ in the sequel. 
\end{rem}
We can now apply the results of the previous subsection to describe
the tangent spaces to $\mathcal{M}_{X,\Omega}^{D}(r,n)$.

\subsection{Infinitesimal study: moduli spaces}

The aim of this subsection is to prove the following theorem.
\begin{thm}
Let $\xi=[E,a,\varphi]\in\mathcal{M}_{X,\Omega}^{D}(r,n)$. The tangent
space $T_{\xi}\mathcal{M}_{X,\Omega}^{D}(r,n)$ is naturally isomorphic
to the kernel of a canonically defined linear map
\[
p_{\varphi}:Ext_{\mathcal{O}_{X}}^{1}(E,E(-D))\rightarrow Ext_{\mathcal{O}_{X}}^{1}(\Lambda^{2}E,\mathcal{O}_{X}(-D)).
\]

\end{thm}
Le us return to the setting and notation of subsection \ref{sub:infparam}.
Consider the natural actions of $Aut(\mathcal{H})$ on $\overset{\circ}{Z^{A}}$
and $\overset{\circ}{Z_{\Omega}^{A}}$. To obtain the tangent spaces
to $\overset{\circ}{Z^{A}}/Aut(\mathcal{H})$ at a point $[(q,\mathcal{A})]$
and to $\overset{\circ}{Z_{\Omega}^{A}}/Aut(\mathcal{H})$ at $[(q,\mathcal{A},\Phi)]$
one has to mod out the image of the induced tangent orbit maps 
\[
End(\mathcal{H})\rightarrow T_{(q,\mathcal{A})}Z^{A},\,End(\mathcal{H})\rightarrow T_{(q,\mathcal{A},\Phi)}Z_{\Omega}^{A}.
\]
These tangent maps factor through 
\[
End(\mathcal{H})\rightarrow Hom(\mathcal{H},\mathcal{E}),\,x\mapsto q\circ x
\]
and are described by 
\[
Hom(\mathcal{H},\mathcal{E})\ni\lambda\mapsto(\lambda\circ\iota,\alpha\circ\lambda)\in T_{(q,\mathcal{A})}Z^{A};
\]
\[
Hom(\mathcal{H},\mathcal{E})\ni\lambda\mapsto(\lambda\circ\iota,\alpha\circ\lambda,\varphi(\lambda\otimes q+q\otimes\lambda))\in T_{(q,\mathcal{A})}Z_{\Omega}^{A}.
\]
In the first case, the quotient can be shown to be isomorphic to the
hyperext group 
\[
\mathbb{E}xt{}^{1}(\mathcal{E},\mathcal{E}\rightarrow\mathcal{D}).
\]
To prove this, one interprets a pair $(\gamma,\mathcal{B})\in T_{(q,\mathcal{A})}Z^{A}$
as a morphism of complexes 
\[
\xymatrix{\mathcal{K}\ar[r]^{\iota}\ar[d]_{\gamma} & \mathcal{H}\ar[d]^{\mathcal{B}}\\
\mathcal{E}\ar[r]_{\mathcal{A}} & \mathcal{D}
}
\]
and it is immediate to see that the subspace of nullhomotopic morphisms
coincides with the image of $Hom(\mathcal{H},\mathcal{E})\rightarrow T_{(q,\mathcal{A})}Z^{A}$.
We obtain a chain of natural isomorphisms 
\[
T_{(q,\mathcal{A})}Z^{A}/Hom(\mathcal{H},\mathcal{E})\cong Hom_{K}(\mathcal{K}\rightarrow\mathcal{H},\mathcal{E}\rightarrow\mathcal{D})\cong\mathbb{E}xt{}^{1}(\mathcal{E},\mathcal{E}\rightarrow\mathcal{D}),
\]
see \cite{HL}. Fix a symplectic form $\varphi:\Lambda^{2}\mathcal{E}\rightarrow\mathcal{O}_{X_{A}}$.
We define a map
\[
p_{\varphi}:Hom_{K}(\mathcal{K}\rightarrow\mathcal{H},\mathcal{E}\rightarrow\mathcal{D})\rightarrow Hom_{K}(\mathcal{K}\wedge\mathcal{H}\rightarrow\Lambda^{2}\mathcal{H},\mathcal{O}_{X_{A}}\rightarrow\mathcal{O}_{D_{A}})
\]
by assigning to a morphism of complexes $(\gamma,\mathcal{B})$ the
morphism

\[
\varphi(\gamma\otimes q):\mathcal{K}\wedge\mathcal{H}\rightarrow\mathcal{O}_{X_{A}},\,\Omega(\alpha q\otimes\mathcal{B}+\mathcal{B}\otimes\alpha q):\Lambda^{2}\mathcal{H}\rightarrow\mathcal{O}_{D_{A}}.
\]
The map $\varphi(\gamma\otimes q)$ is naturally defined on $\mathcal{K}\otimes\mathcal{H}$
but, due to the skew-symmetry of $\varphi,$ it vanishes on the subsheaf
$S^{2}\mathcal{H}\cap(\mathcal{K}\otimes\mathcal{H})$; we use the
same notation for the naturally induced map on the quotient $\mathcal{K}\wedge\mathcal{H}$,
see Lemma \ref{lem:KwedgeH}. $p_{\varphi}$ is well defined since
if $(\gamma,\mathcal{B})$ is homotopic to $0$ and $\lambda:\mathcal{H}\rightarrow\mathcal{E}$
is an homotopy, then $\varphi(\lambda\otimes q+q\otimes\lambda)$
will be an homotopy for $p_{\varphi}(\gamma,\mathcal{B})$. By definition,
$p_{\varphi}(\gamma,\mathcal{B})=0$ if and only if there exists a
morphism $\psi:\Lambda^{2}\mathcal{H}\rightarrow\mathcal{O}_{X_{A}}$
such that $\psi(\iota\otimes1)=\varphi(\gamma\otimes q)$ and $\psi\mid_{D_{A}}=\Omega(\alpha q\otimes\mathcal{B}+\mathcal{B}\otimes\alpha q)$. 
\begin{rem}
$\mathbb{E}xt{}^{1}(\mathcal{E},\mathcal{E}\rightarrow\mathcal{O}_{D}\otimes W)$
is in fact isomorphic to $Ext{}^{1}(\mathcal{E},\mathcal{E}(-D_{A}))$
as $\mathcal{E}$ is locally free on $D_{A}$ by hypothesis.
\end{rem}

\begin{rem}
The natural map
\[
Hom_{K}(\mathcal{K}\wedge\mathcal{H}\rightarrow\Lambda^{2}\mathcal{H},\mathcal{O}_{X_{A}}\rightarrow\mathcal{O}_{D_{A}})\rightarrow\mbox{\ensuremath{\mathbb{H}}}om_{D^{b}}(\mathcal{K}\wedge\mathcal{H}\rightarrow\Lambda^{2}\mathcal{H},\mathcal{O}_{X_{A}}\rightarrow\mathcal{O}_{D_{A}})
\]
 is an isomorphism. \end{rem}
\begin{proof}
We prove surjectivity first. We represent elements of the target group
as roofs in the derived category, i.e. as pairs given by a quasi-isomorphism
$M^{\centerdot}\rightarrow(\mathcal{K}\wedge\mathcal{H}\rightarrow\Lambda^{2}\mathcal{H})$
and a morphism $M^{\centerdot}\rightarrow(\mathcal{O}_{X_{A}}\rightarrow\mathcal{O}_{D_{A}})$,
where $M^{\centerdot}$ is a complex concentrated in degrees zero
and one. We get a morphism of short exact sequences
\[
\xymatrix{M^{0}\ar[r]\ar[d] & M^{1}\ar[r]\ar[d] & \Lambda^{2}\mathcal{\mathcal{E}}\ar@{=}[d]\\
\mathcal{K}\wedge\mathcal{H}\ar[r] & \Lambda^{2}\mathcal{H}\ar[r] & \Lambda^{2}\mathcal{E}
}
\]
Apply $Hom(\_,\mathcal{O}_{X_{A}})$ and use 
\[
Ext^{1}(\Lambda^{2}\mathcal{H},\mathcal{O}_{X_{A}})\cong\Lambda^{2}V\otimes H^{1}(X_{A},\mathcal{O}_{X_{A}}(2m))=0;
\]
get
\[
\xymatrix{Hom(\Lambda^{2}\mathcal{E},\mathcal{O}_{X_{A}})\ar[r]\ar@{=}[d] & Hom(\Lambda^{2}\mathcal{H},\mathcal{O}_{X_{A}})\ar[r]\ar[d] & Hom(\mathcal{K}\wedge\mathcal{H},\mathcal{O}_{X_{A}})\ar[d]\ar@{->>}[r] & Ext^{1}(\Lambda^{2}\mathcal{E},\mathcal{O}_{X_{A}})\ar@{=}[d]\\
Hom(\Lambda^{2}\mathcal{E},\mathcal{O}_{X_{A}})\ar[r] & Hom(M^{1},\mathcal{O}_{X_{A}})\ar[r] & Hom(M^{0},\mathcal{O}_{X_{A}})\ar[r] & Ext^{1}(\Lambda^{2}\mathcal{E},\mathcal{O}_{X_{A}})
}
\]
 We deduce that the natural map 
\[
Hom(\mathcal{K}\wedge\mathcal{H},\mathcal{O}_{X_{A}})\rightarrow Hom(M^{0},\mathcal{O}_{X_{A}})/Hom(M^{1},\mathcal{O}_{X_{A}})
\]
is surjective; in other words we can suppose that the morphism $M^{0}\rightarrow\mathcal{O}_{X_{A}}$
comes from $Hom(\mathcal{K}\wedge\mathcal{H},\mathcal{O}_{X_{A}})$
up to homotopy. To show that also $M^{1}\rightarrow\mathcal{O}_{D_{A}}$
comes from a map $\Lambda^{2}\mathcal{H}\rightarrow\mathcal{O}_{D_{A}}$,
we may proceed as above: apply $Hom(\_,\mathcal{O}_{D_{A}})$ and
use 
\[
Ext^{1}(\Lambda^{2}\mathcal{H},\mathcal{O}_{D_{A}})\cong\Lambda^{2}V\otimes H^{1}(X_{A},\mathcal{O}_{D_{A}}(2m))=0.
\]
This proves surjectivity. 

To prove injectivity, let $f\in Hom_{K}(\mathcal{K}\wedge\mathcal{H}\rightarrow\Lambda^{2}\mathcal{H},\mathcal{O}_{X_{A}}\rightarrow\mathcal{O}_{D_{A}})$
be such that there exists a quasi isomorphism $M^{\centerdot}\rightarrow(\mathcal{K}\wedge\mathcal{H}\rightarrow\Lambda^{2}\mathcal{H})$
whose composition with $c$ admits a homotopy $h:M^{1}\rightarrow\mathcal{O}_{X_{A}}$.
We only need to prove that $h$ factors through another homotopy $\Lambda^{2}\mathcal{H}\rightarrow\mathcal{O}_{X_{A}}$;
this is again achieved by an easy chasing of the diagram above. 
\end{proof}
Since there is an obvious isomorphism 
\[
\mbox{\ensuremath{\mathbb{H}}}om_{D^{b}}(\mathcal{K}\wedge\mathcal{H}\rightarrow\Lambda^{2}\mathcal{H},\mathcal{O}_{X_{A}}\rightarrow\mathcal{O}_{D_{A}})\cong Ext{}^{1}(\Lambda^{2}\mathcal{E},\mathcal{O}_{X_{A}}(-D_{A})),
\]
the previous discussion leads us to conclude that the tangent space
$T_{\Omega}^{1,A}$ to $\overset{\circ}{Z_{\Omega}^{A}}/Aut(\mathcal{H})$
at $[(q,\mathcal{A},\Phi)]$ fits into an exact sequence of vector
spaces
\[
0\rightarrow T_{\Omega}^{1,A}\rightarrow Ext{}^{1}(\mathcal{E},\mathcal{E}(-D_{A}))\rightarrow Ext{}^{1}(\Lambda^{2}\mathcal{E},\mathcal{O}_{X_{A}}(-D_{A})).
\]

\begin{rem}
Let $A=\mathbb{K}$. The map $p_{\varphi}$ is in fact canonical (i.e.,
it only depends on the triple $(E,\varphi,a)$) since it admits the
following Yoneda-type description. Let $\xi\in Ext^{1}(E,E(-D))$
be represented by an extension
\[
\xymatrix{0\ar[r] & E(-D)\ar[r]^{\iota} & F\ar[r]^{\pi} & E\ar[r] & 0}
.
\]
Apply $E\otimes\_$ and pushout via $\varphi(-D):$
\[
\xymatrix{0\ar[r] & ker(1_{E}\otimes\iota)\ar[r]\ar@{->>}[d] & E^{\otimes2}(-D)\ar[r]^{1_{E}\otimes\iota}\ar[d]_{\varphi(-D)} & E\otimes F\ar[d]\ar[r] & E^{\otimes2}\ar[r]\ar@{=}[d] & 0\\
0\ar[r] & ker(\chi)\ar[r] & \mathcal{O}_{X}(-D)\ar[r]_{\chi} & M\ar[r]_{p} & E^{\otimes2}\ar[r] & 0
}
\]
Now, the module $ker(1_{E}\otimes\iota)$ is a torsion sheaf; it is
indeed an epimorphic image of the $0-$dimensional sheaf $\mathcal{T}or_{1}(E,E)$.
It follows that $ker(\chi)$ is torsion as well: it is then forced
to vanish. The above construction defines a linear map
\[
\varphi(1_{E}\otimes\_):Ext^{1}(E,E(-D))\rightarrow Ext^{1}(E^{\otimes2},\mathcal{O}_{X}(-D)).
\]
Similarly, define the ``adjoint'' map $\varphi(\_\otimes1_{E}).$
It is immediate to verify that the map 
\[
\varphi(\_\otimes1_{E})+\varphi(1_{E}\otimes\_):Ext^{1}(E,E(-D))\rightarrow Ext^{1}(E^{\otimes2},\mathcal{O}_{X}(-D))
\]
takes values in the subspace of skew extensions 
\[
Ext^{1}(\Lambda^{2}E,\mathcal{O}_{X}(-D))\subseteq Ext^{1}(E^{\otimes2},\mathcal{O}_{X}(-D)),
\]
and a direct check shows that the equation 
\[
\varphi(\_\otimes1_{E})+\varphi(1_{E}\otimes\_)=p_{\varphi}
\]
holds. \end{rem}
\begin{cor}
\label{thm:Obstruction}Let $(E,\alpha,\varphi)$ be a framed symplectic
sheaf whose underlying framed sheaf $(E,\alpha)$ corresponds to a
smooth point $[(E,\alpha)]\in\mathcal{M}_{X}^{D}(r,n)$, and suppose
that the map $p_{\varphi}$ is an epimorphism. Then $[(E,\alpha,\varphi)]$
is a smooth point of $\mathcal{M}_{X,\Omega}^{D}(r,n)$.\end{cor}
\begin{proof}
Let $A$ be an artinian local $\mathbb{K}-$algebra and let $(\mathcal{E},\mathcal{A},\Phi)\in\mathfrak{M}_{X,\Omega}^{D}(A)$
be a framed symplectic sheaf over $X_{A}$. We proved that the space
of its infinitesimal deformations can be written as 
\[
T^{1}(\mathcal{E},\mathcal{A},\Phi)_{A}=ker(Ext{}_{\mathcal{O}_{X_{A}}}^{1}(\mathcal{E},\mathcal{E}(-D_{A}))\rightarrow Ext{}_{\mathcal{O}_{X_{A}}}^{1}(\Lambda^{2}\mathcal{E},\mathcal{O}_{X_{A}}(-D_{A}))).
\]
We want to give a sufficient condition for the smoothness at a closed
point $(E,\alpha,\varphi)$ by means of the $T^{1}-$lifting property.
In our setting, this property may be expressed in the following way.
Let $A_{n}\cong\mathbb{C}[t]/t^{n+1}$, $n\in\mathbb{N}$. Let $(\mathcal{E}_{n},\mathcal{A}_{n},\Phi_{n})\in\mathfrak{M}_{X,\Omega}^{D}(A_{n})$
and $(\mathcal{E}_{n-1},\mathcal{A}_{n-1},\Phi_{n-1})\in\mathfrak{M}_{X,\Omega}^{D}(A_{n-1})$
be its pullback via the natural map $A_{n}\twoheadrightarrow A_{n-1}$.
We get a map
\[
T^{1}(\mathcal{E}_{n},\mathcal{A}_{n},\Phi_{n})_{A_{n}}\rightarrow T^{1}(\mathcal{E}_{n-1},\mathcal{A}_{n-1},\Phi_{n-1})_{A_{n-1}}
\]
and the underlying closed point $[(E,\alpha,\varphi)]$, where $(E,\alpha,\varphi)=(\mathcal{E}_{n},\mathcal{A}_{n},\Phi_{n})\,mod(t)$,
turns out to be a smooth point of $\mathcal{M}_{X,\Omega}^{D}(r,n)$
if and only if the above map is surjective for any $n$. From the
exact sequence of $\mathcal{O}_{X_{n}}(:=\mathcal{O}_{X_{A_{n}}})$
modules
\[
0\rightarrow\mathcal{O}_{X}\rightarrow\mathcal{O}_{X_{n}}\rightarrow\mathcal{O}_{X_{n-1}}\rightarrow0
\]
we get exact sequences
\[
0\rightarrow E\rightarrow\mathcal{E}_{n}\rightarrow\mathcal{E}_{n-1}\rightarrow0
\]
and

\[
0\rightarrow\Lambda^{2}E\rightarrow\Lambda^{2}\mathcal{E}_{n}\rightarrow\Lambda^{2}\mathcal{E}_{n-1}\rightarrow0.
\]
Making extensive use of Lemma \ref{lem:Spectral tor}, we obtain a
commutative diagram with exact rows and columns

\[
\xymatrix{0\ar[r] & T^{1}\ar[r]\ar[d] & Ext_{\mathcal{O}_{X}}{}^{1}(E,E(-D))\ar[r]^{h}\ar[d] & Ext_{\mathcal{O}_{X}}^{1}(\Lambda^{2}E,\mathcal{O}_{X}(-D))\ar[d]\\
0\ar[r] & T^{1,n}\ar[r]\ar[d]_{a} & Ext{}_{\mathcal{O}_{X_{n}}}^{1}(\mathcal{E}_{n},\mathcal{E}_{n}(-D_{n}))\ar[r]^{p_{n,\varphi}}\ar[d]_{b} & Ext_{\mathcal{O}_{X_{n}}}^{1}(\Lambda^{2}\mathcal{E}_{n},\mathcal{O}_{X_{n}}(-D_{n}))\ar[d]_{c}\\
0\ar[r] & T^{1,n-1}\ar[r] & Ext{}_{\mathcal{O}_{X_{n-1}}}^{1}(\mathcal{E}_{n-1},\mathcal{E}_{n-1}(-D_{n-1})\otimes W)\ar[r] & Ext_{\mathcal{O}_{X_{n-1}}}^{1}(\Lambda^{2}\mathcal{E}_{n-1},\mathcal{O}_{X_{n-1}}(-D_{n-1}))
}
\]
Call $\tilde{c}$ the restriction of $c$ to the image of $p_{n,\varphi}$
in the diagram; by applying the snake lemma to the second and third
row, we get an exact sequence of vector spaces
\[
ker(a)\rightarrow ker(b)\rightarrow ker(\tilde{c})\rightarrow coker(a)\rightarrow coker(b)\rightarrow coker(\tilde{c}).
\]
We know by hypothesis that $coker(b)=0$. A sufficient condition to
get or claim $coker(a)=0$ is $ker(b)\rightarrow ker(c)$ be surjective
$(\implies ker(c)=ker(\tilde{c}))$. This condition is clearly satisfied
if $h$ is surjective. 
\end{proof}
We conclude the section with a direct application to the case of bundles.
\begin{cor}
\label{cor:smooth lf points}If $(E,\alpha,\varphi)$ is a framed
symplectic bundle, the corresponding point $[(E,\alpha,\varphi)]\in\mathcal{M}_{X,\Omega}^{D,reg}(r,n)$
is smooth if $[(E,\alpha)]\in\mathcal{M}_{X}^{D,reg}(r,n)$ is. \end{cor}
\begin{proof}
Let $(E,\alpha,\varphi)$ be a symplectic bundle. Consider the following
map: 
\[
\mathcal{H}om(E,E(-D))\rightarrow\mathcal{H}om(\Lambda^{2}E,\mathcal{O}(-D))\subseteq\mathcal{H}om(E,E^{\vee}(-D))
\]
defined on sections by 
\[
f\mapsto\varphi(-D)\circ f+(f(-D))^{\vee}\varphi,
\]
where $\varphi$ is interpreted as an isomorphism $E\rightarrow E^{\vee}$.
The kernel of this map is identified with the bundle $Ad^{\varphi}(E)(-D)$,
i.e. the twisted adjoint bundle associated with the principal $SP-$bundle
defined by $(E,\varphi)$. The map defined above is easily proved
to be surjective since $\varphi$ is an isomorphism. We obtain a short
exact sequence of bundles which is in fact split-exact, as the map
\[
\mathcal{H}om(\Lambda^{2}E,\mathcal{O}(-D))\rightarrow\mathcal{H}om(E,E(-D)),\,\psi\mapsto\frac{1}{2}\psi(\varphi(-D))^{-1}
\]
gives a splitting. In particular, the $H^{1}-$factors of the corresponding
long exact sequence in cohomology define an exact sequence 
\[
0\rightarrow H^{1}(Ad^{\varphi}(E)(-D))\rightarrow Ext^{1}(E,E(-D))\rightarrow Ext^{1}(\Lambda^{2}E,\mathcal{O}(-D)))\rightarrow0
\]
whose second map is just the map $p_{\varphi}$. The surjectivity
of the latter provides the result. We remark that the result does
not require the natural obstruction space $H^{2}(X,Ad^{\varphi}(-D))$,
from the theory of principal bundles, to vanish.
\end{proof}

\section{Moduli spaces of framed symplectic sheaves on $\mathbb{P}_{\mathbb{C}}^{2}$}

We fix now $\mathbb{K}=\mathbb{C}$. The moduli space of framed sheaves
on $X=\mathbb{P}^{2}$ (where the framing divisor is a line) has a
description in terms of the so called \emph{ADHM data, }which means
that\emph{ }it can be realized as a quotient of certain spaces of
matrices. In the present section we prove that an analogous result
holds in the symplectic case. In the next section, we will apply this
result for proving that our moduli space is irreducible.

\subsection{ADHM data and monads}

Let $r$, $n$ be positive integers, and let $W,$ $V$ be complex
vector spaces with $dim(W)=r$, $dim(V)=n$. 
\begin{defn}
\label{def:ADHM classic}The \emph{variety of ADHM data }of type $(r,n)$
is the closed subvariety of the affine space 
\[
End(V)^{\oplus2}\oplus Hom(W,V)\oplus Hom(V,W)
\]
defined by
\[
\mathbb{M}(r,n)=\{(A,B,I,J)\mid[A,B]+IJ=0\}.
\]
The equation $[A,B]+IJ=0$ will be called \emph{ADHM equation. }
\end{defn}
The group $GL(V)$ acts on $\mathbb{M}(r,n)$ naturally:
\[
g\cdot(A,B,I,J)=(gAg^{-1},gBg^{-1},gI,Jg^{-1}).
\]

\begin{defn}
We call an ADHM datum $(A,B,I,J)$
\begin{itemize}
\item \emph{stable} if there exists no proper subspace $S\subseteq V$ satisfying
$A(S)\subseteq S,$ $B(S)\subseteq S$ and $im(I)\subseteq S$;
\item \emph{co-stable} if there exists no nonzero subspace $S\subseteq V$
satisfying $A(S)\subseteq S,$ $B(S)\subseteq S$ and $ker(J)\supseteq S$.
\end{itemize}
\end{defn}
Denote by $\mathbb{M}^{s}(r,n)$ (resp $\mathbb{M}^{c}(r,n)$) the
subset of $\mathbb{M}(r,n)$ consisting of stable (resp. co-stable)
ADHM data. They are open and invariant subsets. Call $\mathbb{M}^{sc}(r,n)$
their intersection. The motivation to define such spaces lies in the
following theorem.
\begin{thm}
\cite{Na,BM}\label{thm:ADHM=00003Dmoduli} The $GL(V)$ action is
free and locally proper on $\mathbb{M}^{s}(r,n)$. Let $\mathbb{M}^{s}/GL(V)$
be the associated geometric quotient. There exists an isomorphisms
$\mathcal{M}_{\mathbb{P}^{2}}^{l}(r,n)\cong\mathbb{M}^{s}/GL(V)$,
where $l\subseteq\mathbb{P}^{2}$ is a fixed line, which maps isomorphically
$\mathcal{M}^{reg}(r,n)$ onto $\mathbb{M}^{sc}/GL(V)$. $\mathcal{M}(r,n)$
is a smooth connected quasi-projective variety of dimension $2rn$.\end{thm}
\begin{rem}
\label{rem: how to get a fsh out of ADHM}We explain how to associate
a framed sheaf to an ADHM data $(A,B,I,J)$. Fix homogeneous coordinates
$(x:y:z)$ on $\mathbb{P}^{2}$, so that $l=\{z=0\}$. Define the
two maps of coherent sheaves $\alpha=\alpha(A,B,I,J)$ and $\beta=\beta(A,B,I,J)$:
\[
\alpha:\mathcal{O}_{\mathbb{P}^{2}}(-1)\otimes V\rightarrow\mathcal{O}_{\mathbb{P}^{2}}\otimes(V\oplus V\oplus W),\,\alpha=(z\cdot A+x,z\cdot B+y,z\cdot J)^{\top};
\]
\[
\beta:\mathcal{O}_{\mathbb{P}^{2}}\otimes(V\oplus V\oplus W)\rightarrow\mathcal{O}_{\mathbb{P}^{2}}(1)\otimes V,\,\beta=(-z\cdot B-y,z\cdot A+x,z\cdot I).
\]
The equation $[A,B]+IJ=0$ is indeed equivalent to $\beta\circ\alpha=0$,
and the stability and co-stability conditions for the datum correspond
respectively to the surjectivity of $\beta$ and injectivity of $\alpha$
as a map of \emph{vector bundles }(as a map of coherent sheaves, $\alpha$
is automatically injective). Thus, a three-term complex of sheaves,
exact anywhere but in degree $0$, has been associated with an ADHM
datum:
\[
\mathcal{O}_{\mathbb{P}^{2}}(-1)\otimes V\overset{\alpha}{\longrightarrow}\mathcal{O}_{\mathbb{P}^{2}}\otimes(V^{\oplus2}\oplus W)\overset{\beta}{\longrightarrow}\mathcal{O}_{\mathbb{P}^{2}}(1)\otimes V.
\]
Define $E=ker(\beta)/im(\alpha)$. If one restricts the monad to $D$
(i.e. imposes $z=0$), its cohomology defines the trivial bundle $\mathcal{O}_{D}\otimes W$
on $D$. Thus, we get an induced isomorphism $a:E\mid_{D}\rightarrow\mathcal{O}_{D}\otimes W$.
In addition, it turns out that any framed sheaf on $\mathbb{P}^{2}$
can be realized as the cohomology of a complex as above, estabilishing
an isomorphism as stated.
\end{rem}
We shall see that it is possible to interpret framed symplectic sheaves
by means of suitable modification of these ADHM data. We start with
a definition.
\begin{defn}
We define a \emph{monad} $M$ on a scheme $S$ to be a complex of
locally free sheaves
\[
\xymatrix{\mathcal{U}\ar[r]^{\alpha} & \mathcal{W}\ar[r]^{\beta} & \mathcal{T}}
\]
with $\alpha$ injective and $\beta$ surjective. Let $\mathcal{E}(M)=ker(\beta)/im(\alpha)$. \end{defn}
\begin{rem}
Let $\mathcal{T}^{\vee}\overset{\beta^{\vee}}{\rightarrow}\mathcal{W}^{\vee}\overset{\alpha^{\vee}}{\rightarrow}\mathcal{U}^{\vee}$
be the dual complex $M^{\vee}$. This is no longer a monad since $\alpha^{\vee}$
need not be surjective. Nevertheless, we have a natural isomorphism
$H^{1}(M^{\vee})\cong\mathcal{E}^{\vee}$. This is proved using the
so called \emph{display of the monad}, as follows. From the monad
$M$ we construct a commutative diagram with exact rows and columns:
\[
\xymatrix{ & 0\ar[d] & 0\ar[d]\\
 & \mathcal{U}\ar[d]\ar@{=}[r] & \mathcal{U}\ar[d]\\
0\ar[r] & ker(\beta)\ar[d]\ar[r] & \mathcal{W}\ar[d]\ar[r] & \mathcal{T}\ar[r]\ar@{=}[d]\ar[r] & 0\\
0\ar[r] & \mathcal{E}\ar[d]\ar[r] & coker(\alpha)\ar[d]\ar[r] & \mathcal{T}\ar[r] & 0\\
 & 0 & 0
}
\]
Passing to duals, we get a diagram

\[
\xymatrix{ &  & 0\ar[d] & 0\ar[d]\\
0\ar[r] & \mathcal{T}^{\vee}\ar[r]\ar@{=}[d] & coker(\alpha)^{\vee}\ar[d]\ar[r] & \mathcal{E}^{\vee}\ar[r]\ar[d] & 0\\
0\ar[r] & \mathcal{T}^{\vee}\ar[r] & \mathcal{W}^{\vee}\ar[d]\ar[r] & ker(\beta)^{\vee}\ar[r]\ar[d]\ar[r] & 0\\
 &  & \mathcal{U}^{\vee}\ar@{=}[r] & \mathcal{U}^{\vee}
}
\]
which gives $ker(\alpha^{\vee})/im(\beta^{\vee})\cong\mathcal{E}^{\vee}$. 
\end{rem}
In particular, for any morphism of complexes
\[
\xymatrix{\mathcal{U}\ar[r]\ar[d] & \mathcal{W}\ar[d]\ar[r] & \mathcal{T}\ar[d]\\
\mathcal{T}^{\vee}\ar[r] & \mathcal{W}^{\vee}\ar[r] & \mathcal{U}^{\vee}
}
\]
we obtain a morphism $\mathcal{E}\rightarrow\mathcal{E}^{\vee}$. 

The following proposition is a slight generalization of \cite[Lemma 4.1.3]{OSS}. 
\begin{prop}
\label{prop:oko}Fix two complexes
\[
M:\:\xymatrix{\mathcal{U}\ar[r]^{\alpha} & \mathcal{W}\ar[r]^{\beta} & \mathcal{T}}
\]
\[
M^{\prime}:\:\xymatrix{\mathcal{U}^{\prime}\ar[r]^{\alpha^{\prime}} & \mathcal{W}^{\prime}\ar[r]^{\beta^{\prime}} & \mathcal{T}^{\prime}}
\]
where $M$ is a monad and $\alpha^{\prime}$ is injective. Let $\mathcal{E}=H^{1}(M)$
and $\mathcal{E}^{\prime}=H^{1}(M^{\prime})$. Assume that the following
vanishings hold:
\[
Hom_{S}(\mathcal{W},\mathcal{U}^{\prime})=Hom_{S}(\mathcal{T},\mathcal{W}^{\prime})=0
\]
\[
Ext_{S}^{1}(\mathcal{T},\mathcal{U}^{\prime})=Ext_{S}^{1}(\mathcal{W},\mathcal{U}^{\prime})=Ext_{S}^{1}(\mathcal{T},\mathcal{W}^{\prime})=Ext_{S}^{2}(\mathcal{T},\mathcal{U}^{\prime})=0.
\]
Then the natural morphism $Hom(M,M^{\prime})\rightarrow Hom_{S}(\mathcal{E},\mathcal{E}^{\prime})$
is a bijection. \end{prop}
\begin{proof}
Fix a morphism $\phi:\mathcal{E}\rightarrow\mathcal{E}^{\prime}$.
This gives a morphism $ker(\beta)\rightarrow ker(\beta^{\prime})/im(\alpha^{\prime})=\mathcal{E}^{\prime}.$
From the exact sequence 
\[
Hom(ker(\beta),\mathcal{U}^{\prime})\rightarrow Hom(ker(\beta),ker(\beta^{\prime}))\rightarrow Hom(ker(\beta),\mathcal{E}^{\prime})\rightarrow Ext^{1}(ker(b),\mathcal{U}^{\prime})
\]
one deduces $Hom(ker(\beta),ker(\beta^{\prime}))\cong Hom(ker(\beta),\mathcal{E}^{\prime})$,
since $Hom(ker(\beta),\mathcal{U}^{\prime})$ and $Ext^{1}(ker(b),\mathcal{U}^{\prime})$
sit in an exact sequence between $Hom(\mathcal{W},\mathcal{U}^{\prime})$,
$Ext^{1}(\mathcal{T},\mathcal{U}^{\prime})$ and $Ext^{1}(\mathcal{W},\mathcal{U}^{\prime})$,
$Ext^{2}(\mathcal{T},\mathcal{U}^{\prime})$ respectively. These four
spaces vanish by hypothesis. So, $\phi$ lifts uniquely to a morphism
\[
\phi_{1}\in Hom(ker(\beta),ker(\beta^{\prime}))\subseteq Hom(ker(\beta),\mathcal{W}^{\prime}).
\]
As $Hom(\mathcal{T},\mathcal{W}^{\prime})=Ext^{1}(\mathcal{T},\mathcal{W}^{\prime})=0$,
we get $Hom(ker(\beta),\mathcal{W}^{\prime})\cong Hom(\mathcal{W},\mathcal{W}^{\prime})$.
This provides indeed an inverse for $Hom(M,M^{\prime})\rightarrow Hom_{S}(\mathcal{E},\mathcal{E}^{\prime})$.\end{proof}
\begin{rem}
Let $(E,a,\varphi)$ be an $l-$framed symplectic sheaf on $\mathbb{P}^{2}$.
Write $E$ as the cohomology of a monad 
\[
\xymatrix{\mathcal{O}(-1)\otimes V\ar[r]^{\alpha} & \mathcal{O}\otimes(V^{\oplus2}\oplus W)\ar[r]^{\beta} & \mathcal{O}_{\mathbb{P}^{2}}(1)\otimes V}
\]
as in Remk. \ref{rem: how to get a fsh out of ADHM}, $\alpha=(z\cdot A+x,z\cdot B+y,z\cdot J)^{\top}$,
$\beta=(-z\cdot B-y,z\cdot A+x,z\cdot I)$ with $(A,B,I,J)\in\mathbb{M}^{s}(r,n)$.
We note that this monad satisfies the hypothesis of Prop. \ref{prop:oko},
so that we can lift $\varphi$ to a morphism of complexes
\[
\xymatrix{\mathcal{O}_{\mathbb{P}^{2}}(-1)\otimes V\ar[r]^{\alpha}\ar[d]_{G_{1}} & \mathcal{O}_{\mathbb{P}^{2}}\otimes(V\oplus V\oplus W)\ar[r]^{\beta}\ar[d]_{F} & \mathcal{O}_{\mathbb{P}^{2}}(1)\otimes V\ar[d]^{G_{2}}\\
\mathcal{O}_{\mathbb{P}^{2}}(-1)\otimes V^{\vee}\ar[r]_{\beta^{\vee}} & \mathcal{O}_{\mathbb{P}^{2}}\otimes(V^{\vee}\oplus V^{\vee}\oplus W^{\vee})\ar[r]_{\alpha^{\vee}} & \mathcal{O}_{\mathbb{P}^{2}}(1)\otimes V^{\vee}
}
\]

\end{rem}

\begin{rem}
\label{rem:JMW}It is proved in \cite{JMW} that the commutativity
of this diagram, with some additional conditions on $\varphi$ (namely,
skew-symmetry and compatibility with the framing), are equivalent
to the following set of conditions:\end{rem}
\begin{itemize}
\item $G_{2}=-G_{1}:=G$;
\item $F=\begin{pmatrix}0 & G & 0\\
-G & 0 & 0\\
0 & 0 & \Omega
\end{pmatrix};$
\item $GA-A^{\vee}G=0=GB-B^{\vee}G$;
\item $J=-\Omega^{-1}I^{\vee}G$.
\end{itemize}
This motivates the following definition.
\begin{defn}
The \emph{variety of symplectic ADHM data }of type $(r,n)$ is the
closed subvariety of the affine space 
\[
\mathbb{M}_{\Omega}(r,n)\subseteq End(V)^{\oplus2}\oplus Hom(W,V)\oplus Hom(S^{2}V,\mathbb{C})
\]
defined by the equations
\begin{itemize}
\item $GA-A^{\vee}G=0$ \emph{($GA$-symmetry)};
\item $GB-B^{\vee}G=0$ ($GB$-\emph{symmetry});
\item $[A,B]-I\Omega^{-1}I^{\vee}G=0$ \emph{(ADHM equation).}
\end{itemize}
\end{defn}
The group $GL(V)$ acts on $\mathbb{M}_{\Omega}(r,n)$ naturally:
\[
g\cdot(A,B,I,G)=(gAg^{-1},gBg^{-1},gI,g^{-\vee}Gg^{-1}).
\]
To any symplectic ADHM datum one associates a ``classic'' datum,
by defining $J=-\Omega^{-1}I^{\vee}G$:
\[
\iota:\mathbb{M}_{\Omega}(r,n)\rightarrow\mathbb{M}(r,n),\;\iota(A,B,I,G)=(A,B,I,-\Omega^{-1}I^{\vee}G).
\]
The map $\iota$ is clearly a $GL(V)-$equivariant morphism. We call
a symplectic datum stable or co-stable if its associated classic ADHM
datum is, and we denote $\mathbb{M}_{\Omega}^{s}(r,n)$, $\mathbb{M}_{\Omega}^{c}(r,n)$
and $\mathbb{M}_{\Omega}^{sc}(r,n)$ the corresponding open invariant
subsets.
\begin{lem}
\label{lem:Gnonz locf}Let $(A,B,I,G)$ be a stable symplectic datum.
If $S\subseteq V$ is an $A,B-$invariant subspace satisfying $ker(-\Omega^{-1}I^{\vee}G)\subseteq S$,
then $S\subseteq G$. In particular, $(A,B,I,G)$ is co-stable if
and only if $G$ is invertible.\end{lem}
\begin{proof}
Let $S\subseteq ker(-\Omega^{-1}I^{\vee}G)\subseteq V$ be an $A,B-$stable
subspace. Let $s\in S$, $G(s)\in V^{\vee}$. Then $G(s)^{\perp}\supseteq Im(I)$.
Let 
\[
T=\underset{s\in S}{\bigcap}G(s)^{\perp}\subseteq V.
\]
Using the $GA-$symmetry, we prove $T$ is $A-$stable: 
\[
t\in T\implies\left\langle G(s),A(t)\right\rangle =\left\langle A^{\vee}G(s),t\right\rangle =
\]
\[
=\left\langle GA(s),t\right\rangle =0,
\]
since $A(s)\in S$. The same holds for $B$. The stability of the
datum forces $T=V$; but this means $G(S)=0$, i.e. $S\subseteq Ker(G)$. 
\end{proof}

\subsection{ADHM type description for $\mathcal{M}_{\mathbb{P}^{2},\Omega}^{l}$}

The aim of this section is to define a scheme structure on the set
of equivalence classes of symplectic ADHM configurations. We shall
also give a symplectic analogue to Theorem \ref{thm:ADHM=00003Dmoduli},
i.e. we will prove that the resulting scheme is in fact isomorphic
to the moduli space of framed symplectic sheaves on the plane. The
scheme structure is constructed by means of the following lemma.
\begin{lem}
\label{lem:Closem adhm}The restriction of the map $\iota$ to $\mathbb{M}_{\Omega}^{s}$
is a closed embedding.\end{lem}
\begin{proof}
We want to apply Lemma \ref{lem:closem crit}. First we verify that
if $(A,B,I,G_{1})$ and $(A,B,I,G_{2})$ are stable symplectic data
satisfying $I^{\vee}G_{1}=I^{\vee}G_{2}$, then $G_{1}=G_{2}$. Let
$S=Ker(G_{1}-G_{2})$; it is $A,B-$stable by the $GA$ and $GB$
symmetries. Moreover, 
\[
(G_{1}-G_{2})\circ I=[I^{\vee}(G_{1}^{\vee}-G_{2}^{\vee})]^{\vee}=[I^{\vee}(G_{1}-G_{2})]^{\vee}=0,
\]
i.e. $S\supseteq Im(I)$. This forces $S=V$, and proves injectivity
at closed points.

The tangent space $T_{(A,B,I,GJ)}\mathbb{M}^{s}$ can be identified
with the vector space of quadruples 
\[
(X_{A},X_{B},X_{I},X_{J})\in End(V)^{\oplus2}\oplus Hom(W,V)\oplus Hom(V,W)
\]
satisfying the equation
\[
[A,X_{B}]+[X_{A},B]+X_{I}J-IX_{J}=0.
\]
The corresponding description for $T_{(A,B,I,G)}\mathbb{M}_{\Omega}^{s}$
is: 
\[
(X_{A},X_{B},X_{I},X_{G})\in End(V)^{\oplus2}\oplus Hom(W,V)\oplus Hom(S^{2}V,\mathbb{C})
\]
satisfying the equations
\begin{itemize}
\item $GX_{A}-X_{A}^{\vee}G+X_{G}A-A^{\vee}X_{G}=0$;
\item $GX_{B}-X_{B}^{\vee}G+X_{G}B-B^{\vee}X_{G}=0$;
\item \emph{$[A,X_{B}]+[X_{A},B]-X_{I}\Omega^{-1}IG-I\Omega^{-1}X_{I}G-I\Omega^{-1}IX_{G}=0$.}
\end{itemize}
We write the tangent map:
\[
T_{(A,B,I,G)}\mathbb{M}_{\Omega}^{s}\ni(X_{A},X_{B},X_{I},X_{G})\mapsto(X_{A},X_{B},X_{I},-\Omega^{-1}X_{I}^{\vee}G-\Omega^{-1}I^{\vee}X_{G})\in T_{(A,B,I,J)}\mathbb{M}^{s}.
\]
Suppose $(X_{A},X_{B},X_{I},X_{G})\mapsto0$. This forces $X_{G}A-A^{\vee}X_{G}$,
$X_{G}B-B^{\vee}X_{G}=0$ and $I^{\vee}X_{G}=0$, and we may conclude
$(X_{A},X_{B},X_{I},X_{G})=0$ by following the same argument we employed
to prove injectivity on closed ponts (i.e. proving thaht $ker(G)$
is an invariant subspace of $V$ containing $I(W)$).

Eventually, we need to prove properness; we will apply a valuative
criterion as stated in \cite[7.3.9]{GD}. Let $Spec(\mathbb{C}((t)))\rightarrow Spec(\mathbb{C}[[t]])$
be the inclusion of the pointed one-dimensional formal disc into the
formal disc. If for every commutative diagram
\[
\xymatrix{Spec(\mathbb{C}((t)))\ar[r]\ar[d] & \mathbb{M}_{\Omega}^{s}\ar[d]\\
Spec(\mathbb{C}[[t]])\ar[r]\ar@{-->}[ur] & \mathbb{M}^{s}
}
\]
we are able to find a lifting as above, then properness is proved.
This can be rephrased as follows. Suppose $(A_{t},B_{t},I_{t},J_{t})\in\mathbb{M}(\mathbb{C}[[t]])$
such that the pullback to $Spec(\mathbb{C})$, denoted $(A_{0},B_{0},I_{0},J_{0})$,
gives a stable point, and let $G_{t}$ be a symmetric matrix with
entries in $\mathbb{C}((t))$ for which $(A_{t},B_{t},I_{t},G_{t})\in\mathbb{M}_{\Omega}(\mathbb{C}((t)))$
and $J_{t}=-\Omega^{-1}I_{t}^{\vee}G_{t}$. We need to prove that
the entries of $G_{t}$ sit in fact in $\mathbb{C}[[t]]$. Suppose
not. It follows that $G_{t}\neq0$, and one can write $G_{t}=G_{0}t^{-k}+t^{-k+1}G_{1}+\dots$
with $G_{i}$ symmetric $\mathbb{C}-$matrices, $G_{0}\neq0$ and
$k$ a positive integer. We can change the coordinates and put $G_{0}$
in the form 
\[
G_{0}=\begin{pmatrix}1\!\!1_{m} & 0\\
0 & 0_{n-m}
\end{pmatrix},\,m\neq0.
\]
Write $A_{t}=A_{0}+tA_{1}+\dots$ The term of order $-k$ in the equation
$G_{t}A_{t}-A_{t}^{\vee}G_{t}=0$ gives $G_{0}A_{0}-A_{0}^{\vee}G_{0}=0$,
which means that the matrix $A_{0}$ is of the form
\[
A_{0}=\begin{pmatrix}\star & 0\\
\star & \star
\end{pmatrix};
\]
and the same holds for $B_{0}$. Write now $I_{t}=I_{0}+tI_{1}+\dots$,
and let 
\[
I_{0}=\begin{pmatrix}I_{0}^{1}\\
I_{0}^{2}
\end{pmatrix}
\]
 be the block decomposition coherent with our choice of coordinates.
Since $J_{t}=-\Omega^{-1}I_{t}^{\vee}G_{t}$ has entries in $\mathbb{C}[[t]]$,
we get $I_{0}^{\vee}G_{0}=0$ as above. This implies $I_{0}^{1}=0$,
and contradicts the stability of the quadruple $(A_{0},B_{0},I_{0},J_{0})$
since the space of vectors of type $\begin{pmatrix}0\\
v
\end{pmatrix}$ form an $A_{0},B_{0}$ invariant subspace containing in the image
of $I_{0}$. \end{proof}
\begin{rem}
As a consequence of this lemma, we obtain a scheme structure on the
set $\mathbb{M}_{\Omega}^{s}/GL(V)$. Indeed, we have that $\mathbb{M}_{\Omega}^{s}$
is a closed $GL(V)-$invariant subscheme of $\mathbb{M}^{s}$, and
since $\mathbb{M}^{s}/GL(V)$ exists as a geometric quotient as the
action is free and locally proper, we can conclude that the same holds
for $\mathbb{M}_{\Omega}^{s}/GL(V)$. Furthermore, the projection
$\mathbb{M}_{\Omega}^{s}\rightarrow\mathbb{M}_{\Omega}^{s}/GL(V)$
is a $GL(V)-$principal bundle (as it is the pullback of the bundle
$\mathbb{M}^{s}\rightarrow\mathbb{M}^{s}/GL(V)$ by definition).
\end{rem}
The last tool we need to estabilish the desired isomorphism of moduli
spaces is the universal monad on $\mathcal{M}_{\mathbb{P}^{2}}^{l}(r,n)$.
\begin{lem}
\label{lem:Universal monads}\cite[Sect. 7]{He} There exists a monad
\[
\xymatrix{\mathcal{U}\ar[r]^{\alpha} & \mathcal{V}\ar[r]^{\beta} & \mathcal{T}}
\]
on $\mathbb{P}^{2}\times\mathcal{M}_{\mathbb{P}^{2}}^{l}(r,n)$ which
is universal in the following sense: the cohomology of the monad and
its restriction to $l\times\mathcal{M}_{\mathbb{P}^{2}}^{l}(r,n)$
give a pair $(\mathcal{E},a)$ which is a universal framed sheaf for
the moduli space. Let $V,$ $W$ be as in Def. \ref{def:ADHM classic}.
There exists an open affine cover $\{U_{i}\}$ of $\mathcal{M}_{\mathbb{P}^{2}}^{l}(r,n)$
satisfying the following properties:
\begin{itemize}
\item the principal $GL(n)$ bundle $\mathbb{M}^{s}(r,n)\rightarrow\mathcal{M}_{\mathbb{P}^{2}}^{l}(r,n)$
is trivial on $U_{i}$;
\item the restriction of the monad to $\mathbb{P}^{2}\times U_{i}$ looks
like
\[
\xymatrix{\mathcal{O}(-1)\boxtimes\mathcal{O}_{U_{i}}\otimes V\ar[r]^{\alpha_{i}} & \mathcal{O}\boxtimes\mathcal{O}_{U_{i}}\otimes(V^{\oplus2}\oplus W)\ar[r]^{\beta_{i}} & \mathcal{O}_{\mathbb{P}^{2}}(1)\boxtimes\mathcal{O}_{U_{i}}\otimes V}
\]
and one can choose $\alpha_{i}$ and $\beta_{i}$ as follows: let
$\sigma_{i}:U_{i}\rightarrow\mathbb{M}^{s}$ be a section, and define
\[
\alpha_{i}=\alpha\circ\sigma_{i},\,\beta_{i}=\beta\circ\sigma_{i},
\]
see Remk \ref{rem: how to get a fsh out of ADHM}. 
\end{itemize}
\end{lem}
We are finally ready to state and prove the main theorem of this section.
\begin{thm}
There exists an isomorphism of schemes $\mathcal{M}_{\mathbb{P}^{2},\Omega}^{l}\cong\mathbb{M}_{\Omega}^{s}/GL(V)$
which maps $\mathcal{M}_{\mathbb{P}^{2},\Omega}^{l,reg}$ to $\mathbb{M}_{\Omega}^{sc}/GL(V)=\{det(G)\neq0\}/GL(V).$ \end{thm}
\begin{proof}
Let $(\mathcal{E}_{S},\chi_{S},\varphi_{S})$ be an $S-$point of
$\mathcal{M}_{\mathbb{P}^{2},\Omega}^{l}$ for $S$ a given scheme,
i.e. an $S-$flat family of framed symplectic sheaves on $\mathbb{P}^{2}$.
The pair $(\mathcal{E}_{S},\chi_{S})$ induces a morphism $S\rightarrow\mathbb{M}^{s}/GL(V)\cong\mathcal{M}_{\mathbb{P}^{2}}^{l}$.
Let 
\[
M:\;\mathcal{U}\rightarrow\mathcal{W}\rightarrow\mathcal{T}
\]
be a universal monad on $\mathcal{M}_{\mathbb{P}^{2}}^{l}\times\mathbb{P}^{2}$;
if we call $M_{S}$ its pullback to $S\times\mathbb{P}^{2},$ we can
write by definition $\mathcal{E}_{S}=H^{1}(M_{S})$. Choose an affine
open cover $\{S_{i}\}$ of $S$ so that the monad $M_{S_{i}}$ is
of the form
\[
(\mathcal{O}_{\mathbb{P}^{2}}(-1)\boxtimes\mathcal{O}_{S_{i}})\otimes V\rightarrow(\mathcal{O}_{\mathbb{P}^{2}}\boxtimes\mathcal{O}_{S_{i}})\otimes(V^{\oplus2}\oplus W)\rightarrow(\mathcal{O}_{\mathbb{P}^{2}}(1)\boxtimes\mathcal{O}_{S_{i}})\otimes V.
\]
(it is enough to choose an affine refinement of the open cover given
by the preimage of a suitable open cover $\{U_{i}\}$, see Lemma \ref{lem:Universal monads}).
Consider now $\varphi_{S_{i}}:H^{1}(M_{S_{i}})\cong\mathcal{E}_{S_{i}}\rightarrow\mathcal{E}_{S_{i}}^{\vee}\cong H^{1}(M_{S_{i}}^{\vee}).$
Since the pair $(M_{S_{i}},M_{S_{i}}^{\vee})$ satisfies the hypothesis
of Prop. \ref{prop:oko}, we can lift uniquely $\varphi_{S_{i}}$
to a morphism of complexes 
\[
M_{S_{i}}\rightarrow M_{S_{i}}^{\vee},
\]
and if we apply the constraint on the framing and skew-symmetry of
$\varphi_{S}$ as in \ref{rem:JMW}, we get a morphism $f_{i}:S_{i}\rightarrow\mathbb{M}_{\Omega}^{s}$.
On the overlaps $S_{ij}$ we have $f_{i}\sim f_{j}$ under the natural
action of $GL(V)$; therefore, we obtain a map $S\rightarrow\mathbb{M}_{\Omega}^{s}/GL(V)$.
This defines a map $\mathcal{M}_{\mathbb{P}^{2},\Omega}^{l}\rightarrow\mathbb{M}_{\Omega}^{s}/GL(V)$.

We explain how to construct an inverse for this map. Let $(\underline{\mathcal{E}},\underline{\alpha})$
be a universal framed sheaf on $\mathcal{M}_{\mathbb{P}^{2},\Omega}^{l}\times\mathbb{P}^{2}$,
and call $(\mathcal{E},\alpha)$ its restriction to the closed subscheme
$\mathbb{M}_{\Omega}^{s}/GL(V)$. We want to construct a symplectic
form $\Phi$ on $\mathcal{E}$. Choose an open cover of $\mathbb{M}^{s}/GL(V)$
as in Lemma \ref{lem:Universal monads}, and define $\{U_{i}\}$ to
be the induced open cover of the closed subscheme $\mathbb{M}_{\Omega}^{s}/GL(V)$.
The following conditions hold simultaneously:
\begin{enumerate}
\item the pullback of the universal monad to $U_{i}\times\mathbb{P}^{2}$
is of the form
\[
(\mathcal{O}_{\mathbb{P}^{2}}(-1)\boxtimes\mathcal{O}_{U_{i}})\otimes V\rightarrow(\mathcal{O}_{\mathbb{P}^{2}}\boxtimes\mathcal{O}_{U_{i}})\otimes(V^{\oplus2}\oplus W)\rightarrow(\mathcal{O}_{\mathbb{P}^{2}}(1)\boxtimes\mathcal{O}_{U_{i}})\otimes V;
\]

\item the principal bundle $\mathbb{M}_{\Omega}^{s}\rightarrow\mathbb{M}_{\Omega}^{s}/GL(V)$
is trivialized on $U_{i}$. 
\end{enumerate}
We fix sections $\sigma_{i}:U_{i}\rightarrow\mathbb{M}_{\Omega}^{s}$.
Write $\sigma_{i}=(A_{i},B_{i},I_{i},G_{i})$. We obtain a morphism
$G_{i}:U_{i}\rightarrow Hom_{\mathbb{C}}(V,V^{\vee})$, which we use
to define a diagram 
\[
\xymatrix{(\mathcal{O}_{\mathbb{P}^{2}}(-1)\boxtimes\mathcal{O}_{U_{i}})\otimes V\ar[r]\ar[d]_{-G_{i}} & (\mathcal{O}_{\mathbb{P}^{2}}\boxtimes\mathcal{O}_{U_{i}})\otimes(V^{\oplus2}\oplus W)\ar[d]_{\Omega_{i}}\ar[r] & (\mathcal{O}_{\mathbb{P}^{2}}(1)\boxtimes\mathcal{O}_{U_{i}})\otimes V\ar[d]_{G_{i}}\\
(\mathcal{O}_{\mathbb{P}^{2}}(-1)\boxtimes\mathcal{O}_{U_{i}})\otimes V^{\vee}\ar[r] & (\mathcal{O}_{\mathbb{P}^{2}}\boxtimes\mathcal{O}_{U_{i}})\otimes(V^{^{\vee}\oplus2}\oplus W^{\vee})\ar[r] & (\mathcal{O}_{\mathbb{P}^{2}}(1)\boxtimes\mathcal{O}_{U_{i}})\otimes V^{\vee}
}
\]
where 
\[
\Omega_{i}=\begin{pmatrix}0 & G_{i} & 0\\
-G_{i} & 0 & 0\\
0 & 0 & \Omega
\end{pmatrix}.
\]
This yields a morphism of complexes by construction and induces a
collection of morphisms $\varphi_{i}:\mathcal{E}\mid_{U_{i}}\rightarrow\mathcal{E}\mid_{U_{i}}$
which are $GL(V)(S_{i})$-equivalent on the overlaps. We defined a
$\mathbb{M}_{\Omega}^{s}/GL(V)-$family of framed symplectic sheaves
on $\mathbb{P}^{2}$, i.e. a morphism $\mathbb{M}_{\Omega}^{s}/GL(V)\rightarrow\mathcal{M}_{\mathbb{P}^{2},\Omega}^{l}$.
The statement about the locally free locus is a direct consequence
of Lemma \ref{lem:Gnonz locf}. 
\end{proof}

\section{\label{sec:irred}irreducibility of $\mathcal{M}_{\mathbb{P}^{2},\Omega}^{l}$}

The aim of the section is to prove the irreducibility of the moduli
space of framed symplectic sheaves $\mathcal{M}_{\mathbb{P}^{2},\Omega}^{l}(r,n)=:\mathcal{M}_{\Omega}(r,n)$.
We will make use of its description as the orbit space of the action
of $GL(V)$ on the space of stable symplectic ADHM configurations
$\mathbb{M}_{\Omega}^{s}(r,n)$, as explained in the previous section.
We fix Darboux coordinates on the symplectic vector space $(W,\Omega)$
so that $-\Omega^{-1}=\Omega$, $W\cong\mathbb{C}^{r}$. We recall
that the space of quadruples $(A,B,I,G)$ whose orbits correspond
to symplectic bundles are the ones belonging to the open invariant
subset $\{rk(G)=n\}$.
\begin{rem}
The double dual of a symplectic sheaf is a symplectic bundle with
$c_{2}=rk(G)$. We will show how to extract an ADHM datum for $E^{\vee\vee}$
from $E=[A,B,I,G]$. Fix coordinates so that $G=\begin{pmatrix}1\!\!1_{rk(G)} & 0\\
0 & 0_{n-rk(G)}
\end{pmatrix}$. The $GA$, $GB$ symmetries imply that one has 
\[
A=\begin{pmatrix}A^{\prime} & 0\\
a & \alpha
\end{pmatrix},\,B=\begin{pmatrix}B^{\prime} & 0\\
b & \beta
\end{pmatrix}
\]
with $A^{\prime}$ and $B^{\prime}$symmetric. Write $I=\begin{pmatrix}I^{\prime}\\
X
\end{pmatrix}$ according to the above decomposition. We obtain the symplectic ADHM
quadruple $(A^{\prime},B^{\prime},I^{\prime},1\!\!1_{rk(G)})$ which
is a representative for the point $[E^{\vee\vee}]\in\mathcal{M}_{\Omega}^{reg}(r,rk(G)).$
\end{rem}

\subsection*{Strategy of the proof }

The locally free locus $\mathcal{M}_{\Omega}^{reg}(r,n)$ is smooth
of dimension $rn+2n$ and connected, see for example \cite{BFG};
it is then an irreducible variety. The idea is to prove that the closure
of this open subset coincides with the entire $\mathcal{M}_{\Omega}$.
For any given $(A,B,I,G)$, we shall provide a rather explicit construction
of a rational curve in the moduli space passing through $[A,B,I,G]$
and whose general point lies in $\mathcal{M}_{\Omega}^{reg}$. We
shall start studying the cases $G=0$ and $rk(G)=n-1$, as in the
proof for the general case we will use a blend of the techniques for
these two extremal cases.

\subsection{The case $G=0$}

The techniques for this set up are largely inspired from \cite[Appendix A]{Ba1}.
\begin{defn}
A matrix $A\in M_{n}(\mathbb{C})$ is said to be cyclic or nonderogatory
if its minimal polynomial is equal to the characteristic polynomial;
equivalently, there exists a vector $v\in\mathbb{C}^{n}$ such that
$v,Av,\dots A^{n-1}v$ span $\mathbb{C}^{n}$. Clearly, the transpose
of a nonderogatory matrix is again nonderogatory. Further characterizations
of these matrices: \end{defn}
\begin{itemize}
\item $A$ is nonderogatory if and only if any matrix commuting with $A$
is a polynomial in $A$, i.e.
\[
[A,B]=0\implies\exists P\in\mathbb{C}[t]\,\mid\,B=P(A).
\]

\item $A$ is nonderogatory if and only if all of its eigenvalues have geometric
multiplicity equal to $1$ (ex: matrices with no repeated eigenvalues,
Jordan blocks).
\end{itemize}
The space of nonderogatory matrices is a nonempty open subset of $M_{n}$.

We quote now the main result in \cite{TZ}:
\begin{thm}
For any fixed matrix $A$ there exists a nonsingular symmetric matrix
$g$ such that $gAg^{-1}=A^{\top}$. Any matrix $g$ tranforming $A$
into its transpose is symmetric if and only if $A$ is nonderogatory.
\end{thm}
A consequence of this is that any complex matrix $A$ is similar to
a symmetric one. Let $gAg^{-1}=A^{\top}$ with $g$ symmetric. Write
$g=s\cdot s^{\top}$ with $s$ nonsingular. This gives 
\[
s^{\top}As^{-\top}=s^{-1}A^{\top}s=(s^{\top}As^{-\top})^{\top}.
\]

As an immediate corollary one gets:
\begin{cor}
\label{cor:Tz}Suppose $(A,B)$ is a pair of commuting matrices, and
suppose $A$ is nonderogatory. There exists a nonsingular matrix $g$
such that $gAg^{-1}$and $gBg^{-1}$ are symmetric.
\end{cor}
We need another useful result:
\begin{prop}
Let $B$ be any matrix. There exists a nonderogatory matrix $N$ commuting
with $B$.\end{prop}
\begin{proof}
Put $B$ in Jordan form. Perturb each Jordan block by small multiples
of the identity, so that any two distinct blocks are relative to distinct
eigenvalues. The output of this construction is matrix commuting with
$B$, whose eigenvalues all have geometric multiplicity equal to $1$.
\end{proof}
Consider now a framed symplectic sheaf $E$ represented by an ADHM
quadruple $(A,B,I,0)$. Consequently, the ADHM equation reduces to
$[A,B]=0$, and the symmetries are vacuous. Suppose $A$ is nonderogatory.
One can find a basis for which both $A$ and $B$ are symmetric by
Cor. \ref{cor:Tz}. 

Denote the linear subspaces of $M_{n}$ of symmetric and antisymmetric
matrices respectively $S_{n}$ and $AS_{n}$. The linear map 
\[
[A,\_]:S_{n}\rightarrow AS_{n}
\]
 is surjective. Indeed, its kernel has dimension exactly $n=dim(S_{n})-dim(AS_{n})$,
which is the dimension of its subspace defined by polynomials in $A$. 

As a consequence, we can find a symmetric matrix $X$ so that $[A,X]=I\Omega I^{\perp}.$
For $t\in\mathbb{C}$, consider the family of quadruples $(A,B+tX,I,t\cdot Id).$
Since $A$, $B$ and $X$ are symmetric, the $GA,$ $GB$ symmetry
equations are satisfied and, by construction, the $ADHM$ equation
is satisfied as well. For small values of $t$ we can preserve stability,
as it is an open condition. Finally, as $t\cdot Id$ is invertible
for $t\neq0$, we found that $E=[(A,B,I,0)]$ sits in the closure
of $\mathcal{M}_{\Omega}^{reg}$.

We need to show now that the $A-$cyclicity hypothesis may be dropped.
Let $(A,B,I,0)$ as above, and let $N$ be a nonderogatory matrix
with $[N,B]=0$. Consider the line
\[
A_{t}=(1-t)A+tN,\,t\in\mathbb{C}.
\]
This line contains $N,$ and since the set of nonderogatory matrices
is open, one must have that $A_{t}$ is nonderogatoy for any $t,$
except a finite number of values. Furthermore, by construction $[A_{t},B]=0$.
We proved that every neighborhood of $[(A,B,I,0)]$ contains points
of $\overline{\mathcal{M}_{\Omega}^{reg}}$ ; this tells us that $[(A,B,I,0)]\in\overline{\mathcal{M}_{\Omega}^{reg}}$.

\subsection{The case $rk(G)=n-1$}

These ADHM data correspond to symplectic sheaves whose singular locus
is concentrated in one point, with multiplicity 1. We can choose coordinates
such that 
\[
G=\begin{pmatrix}1\!\!1_{n-1} & 0\\
0 & 0
\end{pmatrix}.
\]
The other matrices can be written as 
\[
\begin{pmatrix}A & 0\\
a & \alpha
\end{pmatrix},\,\begin{pmatrix}B & 0\\
b & \beta
\end{pmatrix},\,\begin{pmatrix}I\\
X
\end{pmatrix}.
\]
with $A$ and $B$ symmetric. Without loss of generality we may assume
$\alpha=\beta=0$, since $\mathbb{A}^{2}$ acts on the space of ADHM
configurations by adding multiples of the identity matrix to the endomorphisms.
Of course, the choice of $a$ and $b$ is not unique: by changing
the coordinates we can replace them respectively with $vA+a\lambda,\,vB+b\lambda$
for a nonzero $\lambda\in\mathbb{C}$, keeping $A,B,I$ fixed (the
transformation $g=\begin{pmatrix}1\!\!1_{n-1} & 0\\
v & \lambda
\end{pmatrix}$ does the job). 
\begin{rem}
\label{rem:invsubsp}The subspace of $\mathbb{C}^{n-1}$ 
\[
S=im(A)+im(I)+im(BI)+im(B^{2}I)+\dots+im(B^{n-2}I)
\]
coincides in fact with the whole $\mathbb{C}^{n-1}$. Indeed, the
triple $(A,B,I)$ is stable, and the subspace we are considering contains
by definition the image of $I$, and it is $A$,$B-$invariant. $A-$invariance
is obvious since $S$ contains the image of $A$. To prove $B-$invariance,
we first note that the subspace 
\[
im(I)+im(BI)+im(B^{2}I)+\dots+im(B^{n-2}I)
\]
is $B-$invariant (as $B^{n-1}$ can be written as a polynomial in
$B$ of degree $n-2$ at most). Moreover,  
\[
B(Av)=A(Bv)+I(\Omega I^{\top}v)\,\in im(A)+im(I)
\]
by $[A,B]-I\Omega I^{\top}=0$. 
\end{rem}

\begin{rem}
\label{rem:a=00003D0}If we can choose $a$ (or $b$) to be $0$,
then we can do the following. Consider the family of configurations
\[
(\begin{pmatrix}A & 0\\
0 & 0
\end{pmatrix},\begin{pmatrix}B & tb\\
b & 0
\end{pmatrix},I,\begin{pmatrix}1\!\!1_{n-1} & 0\\
0 & t
\end{pmatrix}).
\]
(or the analogous one with $b=0$). This gives a locally free deformation
of the sheaf $E$. This condition is verified, for example, when $A$
(or $B$) is invertible; in this case, we can write $vA=a$ (or $vB=b$)
and change the coordinates accordingly.
\end{rem}
We need the following technical lemma.
\begin{lem}
\label{lem:tech linal}Let $R$ be a $n-$dimensional vector spaces,
let $T\in End_{\mathbb{C}}(R)$ and $L\subseteq R$ be a subspace.
Let $v\in R$ be $T-$reachable from $L$, meaning
\[
v\in L+TL+T^{2}L+\dots+T^{n-1}L.
\]
There exist parameterized curves $r:\mathbb{C}\rightarrow R$ and
$l:\mathbb{C}\rightarrow L$ with $r(0)=0$, $l(0)=0$ satisfying
\[
(T-t\cdot Id_{R})r(t)=t\cdot v+l(t).
\]
\end{lem}
\begin{proof}
Write $v=\stackrel[i=0]{n-1}{\sum}T^{i}l_{i}$, $l_{i}\in L$. Define
\[
r(t)=\stackrel[i=1]{n-1}{\sum}t^{i}(\stackrel[j=i]{n-1}{\sum}T^{j-i}l_{j}).
\]
We obtain 
\[
(T-tId_{R})r(t)=\stackrel[i=1]{n-1}{\sum}t^{i}(\stackrel[j=i]{n-1}{\sum}T^{j-i+1}l_{j})-\stackrel[i=1]{n-1}{\sum}t^{i+1}(\stackrel[j=i]{n-1}{\sum}T^{j-i}l_{j})=
\]
\[
=t\cdot v-tl_{0}+\stackrel[i=2]{n-1}{\sum}t^{i}(\stackrel[j=i]{n-1}{\sum}T^{j-i+1}l_{j})-\stackrel[i=1]{n-2}{\sum}t^{i+1}(\stackrel[j=i+1]{n-1}{\sum}T^{j-i}l_{j})-\stackrel[i=1]{n-2}{\sum}t^{i+1}l_{i}=
\]
\[
=t\cdot v-\stackrel[i=0]{n-1}{\sum}t^{i+1}l_{i}+\stackrel[i=2]{n-1}{\sum}t^{i}(\stackrel[j=i]{n-1}{\sum}T^{j-i+1}l_{j})-\stackrel[k=2]{n-1}{\sum}t^{k}(\stackrel[j=k]{n-1}{\sum}T^{j-k+1}l_{k})=
\]
\[
=t\cdot v-\stackrel[i=0]{n-1}{\sum}t^{i+1}l_{i}.
\]
So, it is enough to set 
\[
l(t)=-\stackrel[i=0]{n-1}{\sum}t^{i+1}l_{i}.
\]

\end{proof}
We apply the lemma to the following situation. Given an ADHM quadruple
\[
(\begin{pmatrix}A & 0\\
a & 0
\end{pmatrix},\begin{pmatrix}B & 0\\
b & 0
\end{pmatrix},\begin{pmatrix}I\\
X
\end{pmatrix},\begin{pmatrix}1\!\!1_{n-1} & 0\\
0 & 0
\end{pmatrix}),
\]
we set: $R=\mathbb{C}^{n-1}$, $L=im(I)$, $v=a^{\top}$. By Remk.
\ref{rem:invsubsp} we can write 
\[
v=Av_{A}+Ix_{0}+BIx_{1}+\dots+B^{n-1}Ix_{n-1}
\]
for some vector $v_{A}\in R$ and $x_{i}\in W$, and we can find an
equivalent triple with tha same $A$, $B$ and $I$ so that $v=Ix_{0}+BIx_{1}+\dots+B^{n-1}Ix_{n-1}$
(just remember we can move $a$ by any vector in the image of $A$).
Let $r(t)$ and $l(t)=I(Y(t))\in im(I)$ satisfying the thesis, and
write the deformation
\[
(\begin{pmatrix}A & 0\\
a+r(t)^{\top} & 0
\end{pmatrix},\begin{pmatrix}B-tId & 0\\
b & 0
\end{pmatrix},\begin{pmatrix}I\\
X+Y^{\top}(t)\cdot\Omega^{-1}
\end{pmatrix},\begin{pmatrix}1\!\!1_{n-1} & 0\\
0 & 0
\end{pmatrix}).
\]
Now, any point of this curve sits in the space of ADHM configurations,
because $[A,B-tId]=[A,B]$, the $GA$, $GB$ symmetries are obviously
satisfied, and the (2,1) block of the commutator is written as 
\[
(a+r(t)^{\top})(B-tId)-bA=aB-bA-bA+r(t)^{\top}(B-tId)-ta=X\Omega I^{\top}+ta+Y^{\top}(t)\cdot I^{\top}-ta=
\]
\[
=(X+Y^{\top}(t)\Omega^{-1})\Omega I^{\top}.
\]
The previous calculation exhibits a small deformation of the given
configuration which has an invertible matrix ($B-tId$) in the (1,1)
block, and $\beta$ in the (2,2) entry: this must sit in $\overline{\mathcal{M}_{\Omega}^{reg}}$
by Remk \ref{rem:a=00003D0}.

\subsection{The general case}

We are ready to deal with the case of quadruples $(A^{\prime},B^{\prime},I^{\prime},G)$
with $k=rk(G)\in\{1,\dots n-2\},$ where $n=dim(V)$ as usual. We
can normalize $G$ to 
\[
G=\begin{pmatrix}1\!\!1_{k} & 0\\
0 & 0_{n-k}
\end{pmatrix}
\]
and thus write 
\[
A^{\prime}=\begin{pmatrix}A & 0\\
a & \alpha
\end{pmatrix},\,B^{\prime}=\begin{pmatrix}B & 0\\
b & \beta
\end{pmatrix},I^{\prime}=\begin{pmatrix}I\\
X
\end{pmatrix},
\]
with $A,B$ symmetric, $[A,B]=I\Omega I^{\top}$, $[\alpha,\beta]=0$
and $(aB-\beta a)-(bA-\alpha b)=X\Omega I^{\top}$. We note that acting
by the $G-$preserving transformation $g_{v}=\begin{pmatrix}1_{k} & 0\\
v & 1_{n-k}
\end{pmatrix}$ we leave $A$, $B$, $\alpha$, $\beta$ and $I$ untouched and move
$a$ and $b$ respectively to $a+vA-\alpha v$ and $b+vB-\beta v$.
In order to deform our quadruple into a rank $n$ one, we shall need
once again to prove that we can slightly deform it and get a quadruple
with vanishing $a$ or $b$. In the $n-1$ case, this was guaranteed
from $A$ or $B$ being invertible if we have $\alpha=\beta=0$. For
general $\alpha$ and $\beta$, what we should require is $\alpha$
not an eigenvalue for $A$ (or similarly for $\beta$ and $B$). More
generally:
\begin{lem}
\label{lem:Bbeta}Let $S\in Mat(k\times k)$ be symmetric and $\sigma\in Mat((n-k)\times(n-k))$.
Suppose that $S$ and $\sigma$ share no eigenvalues. Then the linear
map 
\[
T\in End(Mat((n-k)\times k)),\:T(v)=vS-\sigma v
\]
is invertible. \end{lem}
\begin{proof}
Choose coordinates on $\mathbb{C}^{n-k}$ so that $\sigma$ is lower
triangular: 
\[
\sigma=\begin{pmatrix}s_{1} & 0 & \cdots\\
\star & s_{2} & \cdots\\
\vdots & \vdots & \ddots
\end{pmatrix}
\]
Suppose there exists a matrix $v=\begin{pmatrix}v_{1}\\
\vdots\\
v_{n-k}
\end{pmatrix}$, $v_{i}\in\mathbb{C}^{k}$ satisfying $vS=\sigma v.$ We have $\sigma v=\begin{pmatrix}s_{1}v_{1}\\
\star\\
\star
\end{pmatrix}=vS=\begin{pmatrix}v_{1}S\\
\vdots\\
v_{n-k}S
\end{pmatrix}.$ We obtain $S(v_{1}^{\top})=s_{1}v_{1}^{\top}$. So if $S$ and $\sigma$
have disjoint spectra, $T$ must be injective, i.e. an isomorphism.
\end{proof}
We want to apply Lemma \ref{lem:tech linal} again to prove that we
can perturb the quadruple to separate the spectra of $B$ and $\beta$
in order to obtain a quadruple for which $b=0$. We need the following
generalization of Remk. \ref{rem:invsubsp}.
\begin{lem}
\label{lem:invsub}Let $R=Mat((n-k)\times k)$ and $A,B,I$ as above
($A$ and $B$ symmetric $k\times k$ matrices, $I\in Hom(\mathbb{C}^{r},\mathbb{C}^{k})$
with $(\mathbb{C}^{r},\Omega)$ a symplectic vector space, $[A,B]=I\Omega I^{\top}$,
stability is satisfied). 
\begin{enumerate}
\item Suppose that there exists a subspace $R^{\prime}\subseteq R$ which
is stable with respect to the maps $v\mapsto vA$, $v\mapsto vB$
and containing the image of the linear map 
\[
\tilde{I}:Hom(\mathbb{C}^{r},\mathbb{C}^{n-k})\rightarrow R,\,\,X\mapsto XI^{\top}.
\]
Then $R^{\prime}=R.$ 
\item Let $\alpha,\beta\in Mat((n-k)\times(n-k))$, $[\alpha,\beta]=0$
and let $T_{A,\alpha},T_{B,\beta}\in End(R)$ defined by 
\[
v\mapsto vA-\alpha v,\,v\mapsto vB-\beta v
\]
respectively. Then the identity
\[
im(T_{A,\alpha})+im(\tilde{I})+im(T_{B,\beta}\tilde{I})+im(T_{B,\beta}^{2}\tilde{I})+\dots+im(T_{B,\beta}^{dim(R)-1}\tilde{I})=R
\]
holds.
\end{enumerate}
\end{lem}
\begin{proof}
The first part is very easy. It is enough to prove that for a given
$R^{\prime}$ as in the hypothesis, any matrix of the form 
\[
v=\begin{pmatrix}0\\
\vdots\\
0\\
v_{i}\\
0\\
\vdots\\
0
\end{pmatrix},\,v_{i}\in\mbox{\ensuremath{\mathbb{C}}}^{k}
\]
belongs to $R^{\prime}$. We will prove it only for $i=1$, as the
other cases are completely analogous. The linear subspace $S\subseteq\mathbb{C}^{k}$
given by vectors $s$ such that $\begin{pmatrix}s\\
0\\
\vdots\\
0
\end{pmatrix}\in R^{\prime}$ must be necessarily the whole of $\mathbb{C}^{k}$, since 
\[
\begin{pmatrix}s\\
0\\
\vdots\\
0
\end{pmatrix}A=\begin{pmatrix}sA\\
0\\
\vdots\\
0
\end{pmatrix}\in R^{\prime},
\]
we see that $S$ is $A-$stable (and similarly, $B-$stable). Furthermore,
\[
X_{1}=\begin{pmatrix}x_{1}\\
0\\
\vdots\\
0
\end{pmatrix}\in Hom(\mathbb{C}^{r},\mathbb{C}^{n-k})\implies\tilde{I}(X_{1})=X_{1}I^{\top}=\begin{pmatrix}x_{1}I^{\top}\\
0\\
\vdots\\
0
\end{pmatrix}\in R^{\prime},
\]
so $S\supseteq im(I)$, and we are done thanks to the stability of
the triple $(A,B,I)$.

To prove the second part, just apply the first to 
\[
R^{\prime}=im(T_{A,\alpha})+im(\tilde{I})+im(T_{B,\beta}\tilde{I})+im(T_{B,\beta}^{2}\tilde{I})+\dots+im(T_{B,\beta}^{dim(R)-1}\tilde{I}).
\]
By definition $R^{\prime}$ contains $Im(\tilde{I})$, so we just
need to prove it is $A$, $B-$ stable in the above sense. We note
that any $v\in im(\tilde{I})+im(T_{B,\beta}\tilde{I})+im(T_{B,\beta}^{2}\tilde{I})+\dots+im(T_{B,\beta}^{dim(R)-1}\tilde{I})$
can be rewritten as 
\[
v=X_{0}^{\prime}I^{\top}+X_{1}^{\prime}I^{\top}B+\dots+X_{k-1}^{\prime}I^{\top}B^{k-1},
\]
because 
\[
T_{B,\beta}^{m}\tilde{I}X=XI^{\top}B^{k}-\begin{pmatrix}m\\
2
\end{pmatrix}\beta XI^{\top}B^{k-1}+\dots-(1)^{m}\beta^{m}XI^{\top}.
\]
So, we may write: 
\[
R^{\prime}=im(T_{A,\alpha})+im(\tilde{I})+im(\tilde{B}\tilde{I})+im(\tilde{B}^{2}\tilde{I})+\dots+im(\tilde{B}^{k-1}\tilde{I})
\]
where $\tilde{B}v=vB$. Now we have: 
\[
\begin{array}{c}
\tilde{A}(T_{A,\alpha}(v))=T_{A,\alpha}(\tilde{A}v)\in R^{\prime}\\
\tilde{A}(\tilde{B}^{k}\tilde{I}(X))=XI^{\top}B^{k}A=XI^{\top}B^{k}A-\alpha XI^{\top}B^{k}+\alpha XI^{\top}B^{k}=T_{A,\alpha}(XI^{\top}B^{k})+\tilde{B}^{k}\tilde{I}(\alpha X)\in R^{\prime}\\
\tilde{B}(T_{A,\alpha}(v))=vAB-\alpha vB=vBA+vI\Omega I^{\top}-\alpha vB=T_{A,\alpha}(vB)+\tilde{I}(vI\Omega)\in R^{\prime}\\
\tilde{B}(\tilde{B}^{k}\tilde{I}(X))\in Im(\tilde{B}^{k+1}\tilde{I})\subseteq R^{\prime}.
\end{array}
\]
This guarantees $A$, $B-$stability of $R^{\prime}$, and concludes
the proof.
\end{proof}
We are now able to apply Lemma \ref{lem:tech linal} to our ADHM quadruple
\[
(\begin{pmatrix}A & 0\\
a & \alpha
\end{pmatrix},\,\begin{pmatrix}B & 0\\
b & \beta
\end{pmatrix},\,\begin{pmatrix}I\\
X
\end{pmatrix},\begin{pmatrix}1\!\!1_{k} & 0\\
0 & 0_{n-k}
\end{pmatrix})
\]
 in the following way. Thanks to Lemma \ref{lem:invsub}, we can write
\[
a=vA-\alpha v+X_{0}I^{\top}+T_{B,\beta}(X_{1}I^{\top})+\dots T_{B,\beta}^{k-1}(X_{k-1}I^{\top}),
\]
and we can change the coordinates to eliminate the addend $vA-\alpha v$.
We apply Lemma \ref{lem:tech linal} to find a small deformation of
$a$, written as $a_{t}=a+v(t)$, satisfying
\[
(T_{B,\beta}-tId)v(t)=ta+Y(t)\Omega I^{\top}\implies(T_{B,\beta}-tId)a_{t}+T_{B,\beta}(a)=Y(t)\Omega I^{\top}
\]
with $Y(0)=0$, and so we obtain
\[
a_{t}(B-tId)-\beta a_{t}-bA+\alpha b=
\]
\[
=aB-\beta a-bA+\alpha b+Y(t)\Omega I^{\top}=(X+Y(t))\Omega I^{\top}.
\]
In other words, the deformation 
\[
(\begin{pmatrix}A & 0\\
a_{t} & \alpha
\end{pmatrix},\,\begin{pmatrix}B-tId & 0\\
b & \beta
\end{pmatrix},\,\begin{pmatrix}I\\
X+Y(t)
\end{pmatrix},\begin{pmatrix}1\!\!1_{k} & 0\\
0 & 0_{n-k}
\end{pmatrix})
\]
still sits in the space of symplectic ADHM data, and remains stable
for small $t.$ Obviously, the spectra of $\beta$ and $B-tId$ are
disjoint for arbitrarily small nonzero values of $t$; therefore,
up to small deformations, we can indeed assume $b=0$, by means of
Lemma \ref{lem:Bbeta} together with the usual change of coordinates
$a\mapsto vA-\alpha v+a$, $b\mapsto vB-\beta v+b$. 

So we can suppose without loss of generality that our quadruple is
of the form
\[
(\begin{pmatrix}A & 0\\
a & \alpha
\end{pmatrix},\,\begin{pmatrix}B & 0\\
0 & \beta
\end{pmatrix},\,\begin{pmatrix}I\\
X
\end{pmatrix},\begin{pmatrix}1\!\!1_{k} & 0\\
0 & 0_{n-k}
\end{pmatrix}).
\]
Now we are only left to apply what we have learned from the rank $0$
case, which is to play with nonderogatory matrices. First, we note
that we can deform $\alpha$ as $\alpha_{t}=(1-t)\alpha+t\eta$ for
any matrix $\eta$ commuting with $\beta$, and this does not harm
the ADHM equations or the symmetries (this is why we made all the
work to get $b=0$, to make sure that no term of type $b\alpha_{t}$
comes to ruin the party). This way, we can suppose that $\alpha$
is nonderogatory up to small deformations, and has no eigenvalues
in common with $A$ (it is enough to choose $\eta$ nonderogatory
and, if necessary, modify $\alpha$ once again by adding small multiples
of $Id_{n-k}$ to slide the eigenvalues). By changing the coordinates
as usual, we obtain a quadruple of type
\[
(\begin{pmatrix}A & 0\\
0 & \alpha
\end{pmatrix},\,\begin{pmatrix}B & 0\\
b & \beta
\end{pmatrix},\,\begin{pmatrix}I\\
X
\end{pmatrix},\begin{pmatrix}1\!\!1_{k} & 0\\
0 & 0_{n-k}
\end{pmatrix}),
\]
which is exactly as before except for a crucial detail: $\alpha$
is nonderogatory. This means that by acting with a change of coordinates
of type $\begin{pmatrix}1\!\!1_{k} & 0\\
0 & g
\end{pmatrix}$, we can suppose that the commuting matrices $\alpha$ and $\beta$
are symmetric. Finally, let $\chi$ be a symmetric matrix satisfying
$[\alpha,\chi]=X\Omega X^{\top}$: it must exist, due to the cyclicity
of $\alpha$. 

Write down the final deformation
\[
(\begin{pmatrix}A & 0\\
0 & \alpha
\end{pmatrix},\,\begin{pmatrix}B & tb^{\top}\\
b & \beta+t\chi
\end{pmatrix},\,\begin{pmatrix}I\\
X
\end{pmatrix},\begin{pmatrix}1\!\!1_{k} & 0\\
0 & t1\!\!1_{n-k}
\end{pmatrix}).
\]
Let us verify that this curve sits in the space of ADHM symplectic
data:
\[
\begin{pmatrix}1\!\!1_{k} & 0\\
0 & t1\!\!1_{n-k}
\end{pmatrix}\begin{pmatrix}A & 0\\
0 & \alpha
\end{pmatrix}-\begin{pmatrix}A & 0\\
0 & \alpha
\end{pmatrix}^{\top}\begin{pmatrix}1\!\!1_{k} & 0\\
0 & t1\!\!1_{n-k}
\end{pmatrix}=
\]
\[
\begin{pmatrix}A-A^{\top} & 0\\
0 & t\alpha-t\alpha^{\top}
\end{pmatrix}=0.
\]
\[
\begin{pmatrix}1\!\!1_{k} & 0\\
0 & t1\!\!1_{n-k}
\end{pmatrix}\begin{pmatrix}B & tb^{\top}\\
b & \beta+t\chi
\end{pmatrix}-\begin{pmatrix}B & tb^{\top}\\
b & \beta+t\chi
\end{pmatrix}^{\top}\begin{pmatrix}1\!\!1_{k} & 0\\
0 & t1\!\!1_{n-k}
\end{pmatrix}=
\]
\[
=\begin{pmatrix}B-B^{\top} & tb^{\top}-tb^{\top}\\
tb-tb & t\beta-t\beta^{\top}+t^{2}\chi-t^{2}\chi^{\top}
\end{pmatrix}=0.
\]
These give the $G$-symmetries. Recall that the ADHM equation for
$t=0$ is given by $\alpha b-bA=X\Omega I^{\top}$, and that $\Omega^{\top}=-\Omega$.
We verify the ADHM equation along the curve:
\[
[\begin{pmatrix}A & 0\\
0 & \alpha
\end{pmatrix},\begin{pmatrix}B & tb^{\top}\\
b & \beta+t\chi
\end{pmatrix}]-\begin{pmatrix}I\\
X
\end{pmatrix}\Omega\begin{pmatrix}I^{\top} & X^{\top}\end{pmatrix}\begin{pmatrix}1\!\!1_{k} & 0\\
0 & t1\!\!1_{n-k}
\end{pmatrix}=
\]
\[
=\begin{pmatrix}[A,B]-I\Omega I^{\top} & tAb^{\top}-tb^{\top}\alpha-tI\Omega X^{\top}\\
\alpha b-bA-X\Omega I^{\top} & t[\alpha,\chi]-tX\Omega X^{\top}
\end{pmatrix}=0
\]

This finishes the proof of the irreducibility, since $G_{t}=\begin{pmatrix}1\!\!1_{k} & 0\\
0 & t1\!\!1_{n-k}
\end{pmatrix}$ is invertible and so the general point of this curve lies in the
locally free locus of the moduli space.
\begin{cor}
$dim(\mathcal{M}_{\Omega}(r,n))=dim(\mathcal{M}_{\Omega}^{reg}(r,n))=rn+2n$.
\end{cor}

\section{Relations with Uhlenbeck spaces and singularities}

The aim of this last section is twofold. First, we want to construct
a proper birational map from the moduli space of framed symplectic
sheaves into the space of symplectic ideal instantons. The map will
simply be the restriction of the so called Gieseker-to-Uhlenbeck map,
defined on the moduli space of framed sheaves (we refer to \cite{BMT}
for the precise definition of this morphism). This provides a concrete
example of the fact that Uhlenbeck spaces of type $C$ can be obtained
by generalized blow-downs of moduli spaces of symplectic sheaves,
as suggested in \cite{Bal,Ba2}. To this end, we will recall the definitions
of Uhlenbeck spaces in a purely algebraic setting. In constrast with
the classical case, the moduli space $\mathcal{M}_{\Omega}(r,n)$
is singular in general. The second part of the section presents some
results about its singular locus.

\subsection{Uhlenbeck spaces}

We recall the definition of Uhlenbeck spaces.
\begin{defn}
Let $(r,n)$ be positive integers. We call \emph{Uhlenbeck space }or
\emph{space of framed ideal instantons }the affine scheme defined
by the categorical quotient 
\[
\mathcal{M}_{0}(r,n):=\mathbb{M}(r,n)/\!/GL(V).
\]

\end{defn}
The following theorem summarizes some of the properties of this scheme.
For details, see \cite{Na,BFG}.
\begin{thm}
The following statements hold:
\begin{enumerate}
\item $\mathcal{M}_{0}(r,n)$ is reduced and irreducible;
\item the open embedding $\mathbb{M}^{sc}(r,n)\hookrightarrow\mathbb{M}(r,n)$
descends to an open embedding $\mathcal{M}^{reg}(r,n)\rightarrow\mathcal{M}_{0}(r,n)$,
which is the smooth locus of $\mathcal{M}_{0}(r,n)$;
\item $\mathcal{M}_{0}(r,n)$ has a stratification into locally closed subsets
of the form
\[
\mathcal{M}_{0}(r,n)=\underset{k=0}{\overset{n}{\bigsqcup}}\mathcal{M}^{reg}(r,n-k)\times(\mathbb{A}^{2})^{(k)},
\]
where $(\mathbb{A}^{2})^{(k)}$ is the $k-th$ symmetric power of
the affine space $\mathbb{P}^{2}\backslash l$;
\item the open embedding $\mathbb{M}^{s}(r,n)\rightarrow\mathbb{M}(r,n)$
descends to a projective morphism 
\[
\pi:\mathcal{M}(r,n)\rightarrow\mathcal{M}_{0}(r,n)
\]
which is a resolution on singularities. 
\end{enumerate}
\end{thm}
\begin{rem}
As a set-theoretic map, $\pi$ has a very simple description. Let
$(E,a)$ be a framed sheaf of charge $n$; the locally free sheaf
$E^{\vee\vee}$ inherits a framing $a^{\vee\vee}$ from $(E,a)$,
and sits in an exact sequence
\[
\xymatrix{0\ar[r] & E\ar[r] & E^{\vee\vee}\ar[r] & C\ar[r] & 0}
\]
where $C$ is a $0-$dimensional sheaf supported away from $l$. Let
$Z(C)$ be the corresponding $0-$cycle on $\mathbb{A}^{2}$. As
\[
length(C)+c_{2}(E^{\vee\vee})=n,
\]
we obtain a point
\[
\pi([E,a])=([E^{\vee\vee},a^{\vee\vee}],Z(C))\in\mathcal{M}_{0}(r,n).
\]

In particular, we see that the restriction of $\pi$ to the locally
free locus $\mathcal{M}^{reg}(r,n)$ induces an isomorphism onto the
open subscheme $\mathcal{M}^{reg}(r,n)\subseteq\mathcal{M}_{0}(r,n)$.
\end{rem}
It is possible to construct a symplectic variant of the Uhlenbeck
space. Once again, let $V\cong\mathbb{C}^{n}$ and $W\cong\mathbb{C}^{r}$.
Fix a symplectic structure $\Omega$ on $W$ and a symmetric bilinear
nondegenerate form $1_{V}$ on $V$, and call $End^{+}(V)$ the space
of symmetric endomorphisms of $V$. Define the subspace of $End^{+}(V)^{\oplus2}\oplus Hom(W,V)$:
\[
\mathbb{X}(r,n)=\{(A,B,I)\mid[A,B]+I\Omega I^{\top}=0\}
\]
This space is naturally acted by the orthogonal group $O(V)$. 
\begin{rem}
Let $\tau:\mathbb{X}(r,n)\rightarrow\mathbb{M}_{\Omega}(r,n)$ be
the embedding defined by 
\[
\tau(A,B,I)\rightarrow(A,B,I,1_{V}).
\]
The morphism $\tau$ is equivariant with respect to the group homomorphism
$O(V)\rightarrow GL(V)$ defined by $1_{V}$. If we set $\mathbb{X}^{s}:\tau^{-1}(\mathbb{M}_{\Omega}^{s})=\tau^{-1}(\mathbb{M}_{\Omega}^{sc})$
(see Lemma \ref{lem:Gnonz locf}), we obtain indeed an isomorphism
of algebraic varieties 
\[
\mathbb{X}^{s}(r,n)/O(V)\rightarrow\mathbb{M}_{\Omega}^{sc}(r,n)/GL(V)(\cong\mathcal{M}_{\Omega}^{reg}(r,n)).
\]
\end{rem}
\begin{defn}
We define the symplectic Uhlenbeck space as the categorical quotient
\[
\mathcal{M}_{0,\Omega}(r,n)=\mathbb{X}(r,n)/\!/O(V).
\]

\end{defn}
We list some interesting properties of this affine scheme. For details,
see \cite{BFG,NS,Ch}.
\begin{thm}
The following statements hold:
\begin{enumerate}
\item $\mathcal{M}_{0,\Omega}(r,n)$ is reduced and irreducible;
\item the open embedding $\mathbb{X}^{s}(r,n)\hookrightarrow\mathbb{X}(r,n)$
descends to an open embedding $\mathcal{M}_{\Omega}^{reg}(r,n)\rightarrow\mathcal{M}_{0,\Omega}(r,n)$,
which is the smooth locus of $\mathcal{M}_{0,\Omega}(r,n)$;
\item $\mathcal{M}_{0,\Omega}(r,n)$ has a stratification into locally closed
subsets of the form
\[
\mathcal{M}_{0,\Omega}(r,n)=\underset{k=0}{\overset{n}{\bigsqcup}}\mathcal{M}_{\Omega}^{reg}(r,n-k)\times(\mathbb{A}^{2})^{(k)};
\]

\item the composition 
\[
\xymatrix{\mathbb{X}(r,n)\ar[r]^{\tau} & \mathbb{M}_{\Omega}(r,n)\ar[r]^{\iota} & \mathbb{M}(r,n)}
\]
is an equivariant closed embedding, inducing a closed embedding 
\[
\mathcal{M}_{0,\Omega}(r,n)\rightarrow\mathcal{M}_{0}(r,n)
\]
which is compatible with the stratifications, meaning that for any
integer $k$
\[
(\mathcal{M}^{reg}(r,n-k)\times(\mathbb{A}^{2})^{(k)})\cap\mathcal{M}_{0,\Omega}(r,n)=\mathcal{M}_{\Omega}^{reg}(r,n-k)\times(\mathbb{A}^{2})^{(k)}
\]
holds.
\end{enumerate}
\end{thm}
As a a direct consequence of the discussion of the previous two sections,
we find a relation between symplectic Uhlenbeck spaces and moduli
of framed symplectic sheaves.
\begin{thm}
The moduli space $\mathcal{M}_{\Omega}(r,n)$ is isomorphic to the
strict transform of the closed subscheme $\mathcal{M}_{0,\Omega}(r,n)\subseteq\mathcal{M}_{0}(r,n)$
under the resolution $\pi$.\end{thm}
\begin{proof}
The maximal open subset of $\mathcal{M}_{0}(r,n)$ over which $\pi$
is an isomorphism is $\mathcal{M}^{reg}(r,n);$ it follows by definition
that the strict transform in the statement is defined as 
\[
\overline{\pi^{-1}(\mathcal{M}_{0,\Omega}(r,n)\cap\mathcal{M}^{reg}(r,n))}=\overline{\pi^{-1}(\mathcal{M}_{\Omega}^{reg}(r,n))}=\mathcal{M}_{\Omega}(r,n),
\]
since $\mathcal{M}_{\Omega}(r,n)$ is the smallest closed subscheme
of $\mathcal{M}(r,n)$ containing the locus of symplectic bundles,
see Sect. \ref{sec:irred}.\end{proof}
\begin{rem}
We note that $\mathcal{M}_{\Omega}(r,n)$ also coincides with the
total transform $\pi^{-1}(\mathcal{M}_{0,\Omega}(r,n))$, as its points
are exactly the framed sheaves whose double dual is symplectic; indeed,
if 
\[
\varphi^{\prime}:E^{\vee\vee}\rightarrow E^{\vee\vee\vee}=E^{\vee}
\]
is a ``honest'' symplectic form on the double dual, its composition
with $E\rightarrow E^{\vee\vee}$ endows $E$ with a structure of
framed symplectic sheaf.
\end{rem}

\subsection{Some remarks on the singularities of $\mathcal{M}_{\mathbb{P}^{2},\Omega}(r,n)$.}
\begin{rem}
$\mathcal{M}_{\mathbb{P}^{2},\Omega}(r,1)$ is smooth for any even
$r$. 
\begin{proof}
We know 
\[
\mathcal{M}_{\mathbb{P}^{2},\Omega}(r,1)\cong\mathbb{M}_{\Omega}^{s}(r,1)/\mathbb{C}^{\star}
\]
where $\mathbb{M}_{\Omega}^{s}(r,1)$ is just the affine variety $\mathbb{C}^{2}\times(\mathbb{C}^{r}\backslash\{0\})\times\mathbb{C}$;
indeed, the symmetries and the ADHM equation are vacuous in this case.
Since the $\mathbb{C}^{\star}-$action is free on this space, no singularities
can arise in the quotient.
\end{proof}
\end{rem}
The case $n=1$ is special, since the singular locus of the Uhlenbeck
space $\mathcal{M}_{0,\Omega}(r,1)$ is smooth (therefore, a single
blow-up is enough to resolve the singularities). In fact, singularities
appear in $\mathcal{M}_{\Omega}(r,n)$ as we take $n>1$. In the next
example, we draw our attention to the case $n=2$. 
\begin{example}
Let $I:W\rightarrow\mathbb{C}^{2}$ be a linear map, and consider
the ADHM quadruple $\xi=(0,0,I,0)\in\mathbb{M}_{\Omega}(r,2).$ This
configuration will be stable if we choose $I$ to be surjective. Let
us compute the dimension of $T_{\xi}\mathbb{M}_{\Omega}^{s}(r,1)$:
using the description of the tangent space in the proof of \ref{lem:Closem adhm},
we have
\[
T_{\xi}\mathbb{M}_{\Omega}^{s}(r,1)=\{(X_{A},X_{B},X_{I},X_{G})\mid I\Omega I^{\top}X_{G}=0\}.
\]
If we choose $I$ so that $I\Omega I^{\top}$ is not invertible (it
is enough to require the rows of the matrix $I$ to span an isotropic
subspace of $W$, and this can be done while keeping $I$ surjective
if $r\geq4$), we obtain 
\[
dim(T_{\xi}\mathbb{M}_{\Omega}^{s}(r,1))>2\cdot2^{2}+2r=8+2r.
\]
The dimension of $\mathbb{M}_{\Omega}^{s}(r,1)$ is exactly $8+2r,$
so we found a singular point.
\end{example}
Unfortunately, Cor. \ref{thm:Obstruction} has no easy translation
into the language of of ADHM data, but we can still say something
on the nonsingular locus of $\mathcal{M}_{\Omega}(r,n)$. 
\begin{rem}
Suppose that a point $x=[(E,\alpha,\varphi)]\in\mathcal{M}_{\Omega}(r,n)$
satisfies the hypothesis of Cor. \ref{thm:Obstruction}. Let $\iota(x)\in\mathcal{M}(r,n)$
be the corresponding framed sheaf. We have an exact sequence of vector
spaces
\[
\xymatrix{0\ar[r] & T_{x}\mathcal{M}_{\Omega}(r,n)\ar[r] & T_{\iota(x)}\mathcal{M}(r,n)\ar[r] & Ext_{\mathcal{O}_{X}}^{1}(\Lambda^{2}E,\mathcal{O}_{X}(-1))\ar[r] & 0}
\]
This forces 
\[
ext^{1}(\Lambda^{2}E,\mathcal{O}_{X}(-1))=2nr-rn-2n=rn-2n.
\]
Now, by Serre duality, we have 
\[
Ext_{\mathcal{O}}^{1}(\Lambda^{2}E,\mathcal{O}_{X}(-1))\cong H^{1}(\mathbb{P}^{2},\Lambda^{2}E(-2)).
\]
Let $E^{\prime}$ be a framed bundle with the same numerical invariants
of $E$. The bundle $\Lambda^{2}E^{\prime}$ is a framed bundle as
well, and we can compute: 
\[
c_{2}(E^{\prime})=H^{1}(\mathbb{P}^{2},\Lambda^{2}E^{\prime}(-2))=-\chi(\Lambda^{2}E^{\prime}(-2))=nr-2n,
\]
see \cite[Sect. 2.1]{Na}. We deduce 
\[
-\chi(\Lambda^{2}E(-2))=nr-2n
\]
since this quantity only depends on the numerical invariants. Let
$T\subseteq\Lambda^{2}E$ be the torsion subsheaf. $T$ is concentrated
on points since $E$ is locally free away from a codimension two locus.
Call $F=\Lambda^{2}E/T$; $F$ is again a framed sheaf, and from the
long exact sequence in cohomology coming from the short exact sequence
\[
\xymatrix{0\ar[r] & T\ar[r] & \Lambda^{2}E(-2)\ar[r] & F(-2)\ar[r] & 0}
\]
we get 
\[
\chi(\Lambda^{2}E(-2))=length(T)-h^{1}(\Lambda^{2}E(-2)).
\]
In order to obtain $\chi(\Lambda^{2}E(-2))=ext^{1}(\Lambda^{2}E,\mathcal{O}_{X}(-1))$
as desired, we are forced to ask $T=0$. Now, the torsion of the seaf
$E^{\otimes2}$ vanishes if and only if $E$ is locally free (see
\cite[Thm. 4]{Au}); we deduce that the hypothesis of \ref{thm:Obstruction}
can hold if and only if the torsion is concentrated in the summand
$S^{2}E$. 

The previous discussion implies that no symplectic sheaf $E$ whose
residue $C=E^{\vee\vee}/E$ is a skyscraper sheaf $\mathbb{C}_{x}$
(i.e. $E$ has a unique singular point with multiplicity $1$) can
satisfy the hypothesis. We sketch a proof of this fact. First, we
can suppose $x=(0,0)\in\mathbb{A}^{2}$. Passing to stalks on $x,$
we can write $E_{x}$ as the kernel of a quotient 
\[
\mathcal{O}^{\oplus r}\rightarrow\mathbb{C}_{(0,0)}.
\]
Denote by $m$ the ideal $(x,y)\subseteq\mathbb{C}[x,y]$. The element
\[
x\wedge y\in\Lambda^{2}m
\]
can be proved very easily to be nonzero and torsion; from this we
get that $\Lambda^{2}E_{x}$ has torsion, and this implies that $\Lambda^{2}E$
has torsion as well. 
\end{rem}
Anyway, the following proposition will show that \ref{thm:Obstruction}
does not give necessary conditions for a point to be smooth.
\begin{prop}
Let $\xi\in\mathcal{M}_{\Omega}(r,n)$ be represented by an ADHM quadruple
$(A,B,I,G)$ such that either $A$ or $B$ is a nonderogatory endomorphism.
Then $\xi$ is a nonsingular point. 
\end{prop}
Indeed, using the constructions in Section \ref{sec:irred} it is
immediate to produce examples of symplectic sheaves whose residue
has length equal to $1$ and which are represented by quadruples with
$A$ nonderogatory. We now prove the proposition.
\begin{proof}
We use once again the ADHM-theoretic description of the tangent space
as exposed in the proof of Lemma \ref{lem:Closem adhm}, i.e. we think
of $T_{(A,B,I,G)}\mathbb{M}_{\Omega}^{s}(r,n)$ as the kernel of the
the Jacobian matrix 
\[
\begin{pmatrix}X_{A}\\
X_{B}\\
X_{I}\\
X_{G}
\end{pmatrix}\mapsto\begin{pmatrix}G\_-\_^{\top}G & 0 & 0 & \_A-A^{\top}\_\\
0 & G\_-\_^{\top}G & 0 & \_B-B^{\top}\_\\
{}[\_,B] & [A,\_] & (I\Omega\_^{\top}+\_\Omega I^{\top})G & I\Omega I^{\top}\_
\end{pmatrix}\begin{pmatrix}X_{A}\\
X_{B}\\
X_{I}\\
X_{G}
\end{pmatrix}
\]
\[
\begin{pmatrix}X_{A}\\
X_{B}\\
X_{I}\\
X_{G}
\end{pmatrix}\in End(V)^{\oplus2}\oplus Hom(W,V)\oplus Hom(S^{2}V,\mathbb{C}).
\]
Suppose $A$ is nonderogatory. We change coordinates to make $A$
symmetric, so that we can write $\_A-A^{\top}\_=-[A,\_]$. The ciclicity
of $A$ guarantees the surjectivity of the map $[A,\_]$ as a morphism
on the space of symmetric matrices with values in the space of skew-symmetric
matrices. If we think $[A,\_]$ as a linear map on the space of general
square matrices, its rank is $n^{2}-n$. We deduce immediately that
the rank of the Jacobian is greater or equal to $\frac{3}{2}n(n-1)$,
hence 
\[
dim(T_{(A,B,I,G)}\mathbb{M}_{\Omega}^{s}(r,n))\leq2n^{2}+nr+\frac{1}{2}n(n+1)-\frac{3}{2}n(n-1)=
\]
\[
=n^{2}+nr+2n=dim(\mathcal{M}_{\Omega}(r,n))+n^{2}\leq dim(T_{(A,B,I,G)}\mathbb{M}_{\Omega}^{s}(r,n)).
\]
This implies that the equation 
\[
dimT_{\xi}(\mathcal{M}_{\Omega}(r,n))=dim(\mathcal{M}_{\Omega}(r,n))
\]
holds, as required.
\end{proof}
The locus of sheaves $[(A,B,I,G)]$ with $A$ nonderogatory is a nonempty
open subscheme of codimension greater or equal than $1$. Indeed,
its complement is contained in the locus cut out by the discriminant
of the characteristic polynomial of $A$ (call it $\Delta_{A}$; it
is a $GL(V)-$invariant function). Any matrix $A$ whose discriminant
is nonzero will be nonderogatory, since in this case $A$ does not
admit double eigenvalues. In particular, the singular locus is contained
in the intersection 
\[
\{\Delta_{A}=0\}\cap\{\Delta_{B}=0\}\cap\{det(G)=0\}.
\]
This intersection cannot be of codimension $1$ unless these divisors
all share some irreducible components, and they do not (to see this,
it is again enough to play with deformations as in Section \ref{sec:irred}).
We conclude:
\begin{cor}
$\mathcal{M}_{\Omega}(r,n)$ is nonsingular in codimension $2$.\end{cor}

\end{document}